\documentclass[reqno]{amsart}
\usepackage{amssymb,latexsym,amsmath,amsthm,enumerate}
\usepackage[mathscr]{eucal}
\usepackage{framed,color,graphicx}
\usepackage{mathrsfs}
\usepackage[all]{xy} 
\usepackage{tikz}
\usetikzlibrary{positioning}

\DeclareSymbolFont{bbold}{U}{bbold}{m}{n}
\DeclareSymbolFontAlphabet{\mathbbold}{bbold}

\theoremstyle{plain} 
\newtheorem{thm}{Theorem}[section] 
\newtheorem{lem}[thm]{Lemma}
\newtheorem{cor}[thm]{Corollary}
\newtheorem{pro}[thm]{Proposition}

\newtheorem{eg}[thm]{Example}

\newtheorem{problem}[thm]{Problem}

\theoremstyle{definition}

\newtheorem{remark}[thm]{Remark}

\newcommand{\logspace}{\texttt{logspace}}

\newcommand{\ga}{\operatorname{ga}}

\newcommand{\up}[1]{\textup{#1}}

\newcommand{\CSP}{\operatorname{CSP}}

\newcommand{\onto}{\twoheadrightarrow}

\newcommand{\ba}{\mathbf{a}}
\newcommand{\bb}{\mathbf{b}}

\newcommand{\ha}{\hat{a}}
\newcommand{\hb}{\hat{b}}

\newcommand{\he}{\hat{e}}

\tikzset{%
element/.style={draw, shape=circle, fill=white, inner sep=1.4pt}}

\begin{document}

\title[Flexible Satisfiability]{Flexible constraint satisfiability and a problem in semigroup theory}

\author{Marcel Jackson}
\address{Department of Mathematical and Physical Sciences\\ La Trobe University\\ Victoria  3086\\
Australia} \email{M.G.Jackson@latrobe.edu.au}

\thanks{The author was supported by ARC Discovery Project DP1094578 and ARC Future Fellowship FT120100666}

\begin{abstract}
We examine some flexible notions of constraint satisfaction, observing some relationships between model theoretic notions of universal Horn class membership and robust satisfiability.  We show the \texttt{NP}-completeness of $2$-robust positive 1-in-3SAT in order to give very small examples of finite algebras with \texttt{NP}-hard variety membership problem.  In particular we give a $3$-element algebra with this property, and solve a widely stated problem by showing that the $6$-element Brandt monoid has \texttt{NP}-hard variety membership problem.  These are the smallest possible sizes for a general algebra and a semigroup to exhibit \texttt{NP}-hardness for the membership problem of finite algebras in finitely generated varieties.
\end{abstract}

\maketitle


There are a number computational situations where the task is not to find if there is a single solution to an instance, but rather to find whether there is a sufficiently broad family of solutions to witness various separation conditions on variables.  The following example is illustrative of idea.  

Recall that a \emph{$3$-colouring} of a graph $\mathbb{G}$ is a homomorphism from $\mathbb{G}$ to the complete graph $\mathbb{K}_3$.  Equivalently, there is a map to the set {of ``colours''} $\{0,1,2\}$ such that adjacent vertices have different colours.  Graph colouring places no restriction on what is to happen to \emph{non-adjacent} vertices.  What happens if we ask that a graph not simply be $3$-colourable, but $3$-colourable in enough ways that there is no restriction to how any pair of nonadjacent vertices may be coloured:  can they  be coloured the same by some valid $3$-colouring, and can they be coloured differently?  This interesting combinatorial question can be rephrased in the following equally interesting way: is it true that every valid colouring of two vertices extends to a full $3$-colouring?

In the present article we look at a range of constraint problems of this kind, grouped loosely under the name  \emph{flexible satisfaction}, with specific kinds of flexible satisfaction to be given more precise definition.  The idea appears natural enough in its own right, and so it is perhaps not surprising that it might emerge from several different investigations of study.  For example, there has been extensive exploration of the emergence of computationally challenging instances in randomly generated instances of constraint problems; see Hayes \cite{hay} or Monasson et al.~\cite{MZKST} for example.  Such work points toward the emergence of hidden constraints being intimately related to computational difficulty; see Beacham and Culberson \cite{beacul} or Culberson and Gent \cite{culgen} for example.  These hidden constraints correspond exactly to non-constrained tuples being preserved as a constraint under any possible solution, or being forced to fall outside of any relation  under any solution (cf.~non-adjacent vertices being forced to be  differently coloured or forced to be like-coloured, respectively).   As a second example, the resolution of problems relating to minimal constraint problems (in the sense of Montenari~\cite{mon}), led Gottlob~\cite{got} to \emph{supersymmetric} SAT instances, which concerns precisely the problem of extendability of arbitrary partial solutions (on a bounded number of variables) to full solutions.  This is developed further in Abramsky, Gottlob and Kolaitis \cite{AGK} in the guise of \emph{robust satisfiability}, where the flexible graph $3$-colouring problem described above is shown to be \texttt{NP}-complete, and related notions are tied to a problem in quantum mechanics.  A third point of interest lies in the fundamental relationship between flexible satisfaction and universal Horn classes: the flexibly $3$-colourable graphs (in the sense described above) are precisely those in the universal Horn class generated by $\mathbb{K}_3$; see Trotta \cite{tro} for example.  

A final point of interest arises somewhat more coincidentally.  Intuitively, constraint networks that can be satisfied in very flexible ways lie at some extreme antipodal point compared to those that are not even satisfiable at all.   Yet in the course of proving \texttt{NP}-completeness, it seems necessary to use a string of reductions which leads to an interesting stretching of the landscape of instances.  It is found that there is no polynomial time boundary separating the ostensibly extreme subclasses \emph{not satisfiable} and \emph{flexibly satisfiable}.  This can already be seen in the results of Abramsky, Gottlob and Kolaitis \cite{AGK}, where after reduction, one arrives at a graph that is either not $3$-colourable at all, or which is flexibly $3$-colourable.  It is also built into the definition of supersymmetric SAT in \cite{got}.  This extreme separation of YES instances from NO instances gives the results extra applicability, which we make use of in the present article.  Aside from wider applicability, the issue is also of interest from a purely computational complexity perspective, as it is at the very least quite reminiscent of the currently open problem as to whether or not there are two disjoint \texttt{NP} languages that are inseparable by a polynomial time boundary (see \cite{FLM, GSSZ,grosel}, for example), and is a natural example of a promise problem (see the survey by Goldreich \cite{gol} for example).  {Interest in promise problems variants of constraint problems has grown substantially in recent years, stemming from  the  \emph{promise CSP} framework introduced by Brakensiek and Guruswami \cite{bragur}.  The promise problem contributions in the present article appear to be somewhat orthogonal to those of \cite{bragur}, but they are perhaps closer to cousins than total strangers, and we able to make use of one of the recent developments in promise CSP to prove Proposition~\ref{pro:highercolouring}.}

A key motivation for the present article comes from computational problems arising in the theory of varieties of algebras.  Recall that the \emph{variety} {generated by} an algebra $\mathbf{A}$ is the class $\mathsf{V}(\mathbf{A})$ of algebras $\mathbf{B}$ of the same signature as $\mathbf{A}$ and which satisfy all identities true on $\mathbf{A}$.  Equivalently, $\mathbf{B}$ is in the variety generated by $\mathbf{A}$ if and only if it is a homomorphic image of a subalgebra of a direct power of $\mathbf{A}$.  The pseudovariety generated by $\mathbf{A}$ is the finite part of the variety generated by $\mathbf{A}$, though in general a \emph{pseudovariety} will be any class of finite algebras of the same signature and which is closed under taking homomorphic images, subalgebras and finitary direct products.  Not every pseudovariety is the finite part of a variety, so the pseudovariety setting is a more encompassing setting for considering finite membership problems of finite algebras.  

Bergman and Slutzki \cite{berslu} gave a $2$-\texttt{exptime} algorithm for deciding membership of~$\mathbf{B}$ in the pseudovariety {generated by} $\mathbf{A}$ (even if both algebras are allowed to vary), while Kozik~\cite{koz09} showed that even if $\mathbf{A}$ is fixed, this is best possible in general: there is a finite algebra with $2$-\texttt{exptime}-complete finite membership problem for its pseudovariety.  Other examples with nontractable finite membership problem (subject to usual complexity theoretic assumptions such as $\texttt{P}\neq \texttt{NP}$) include the original example, due to Szekely \cite{sze}, of a 7-element (or $6$-element \cite{jacmck}) algebra with \texttt{NP}-complete membership problem, semigroups of the author and McKenzie~\cite{jacmck} with \texttt{NP}-hard membership problem (the smallest has 55 elements, though for a monoid one needs 56 elements), a finite algebra due to Kozik with \texttt{PSPACE}-complete~\cite{koz07} finite membership problem for its pseudovariety, and a family of finite semigroups recently discovered Kl\'{\i}ma, Kunc and Pol\'ak~\cite{KKP}, each with co-\texttt{NP}-complete finite membership problem for their pseudovariety (the smallest is a monoid of size 42).  

All of these examples are somewhat ad hoc.  We are able to use the hardness of various flexible satisfaction ideas---and in particular, the extreme separation properties that arise---to give an interval in the lattice of semigroup pseudovarieties in which every pseudovariety has \texttt{NP}-hard finite membership problem.  At the base of this interval is the pseudovariety generated by the ubiquitous six element Brandt monoid ${\bf B}_2^1$.  The particular case of ${\bf B}_2^1$ solves---subject to the assumption $\texttt{P}\neq \texttt{NP}$---Problem~4 of Almeida \cite[p.~441]{alm} and Problem~3.11 of Kharlampovich and Sapir~\cite{khasap}; see also Volkov, Gol$'$dberg and Kublanovski\u{\i} \cite[p.~849]{VGK} (English version).   Also, because every semigroup of order less than $6$ has a finite identity basis (hence tractable membership problem for its pseudovariety), it shows that the smallest size of generator for a pseudovariety with computationally hard membership is~$6$.  For general algebras, we show that there is a $3$-element algebra with \texttt{NP}-complete membership problem for its pseudovariety, also the smallest possible size; we give a $4$-element groupoid (that is, an algebra with a single binary operation) with the same property.  

\subsection*{Results and structure}
The structure of the article is as follows.  In Part~\ref{part:1} we give preliminary discussion of basic concepts relating to constraint satisfaction problems, and relate various flexible notions of satisfaction to model-theoretic classes such as universal Horn classes.  Results here are mostly basic observations, but nevertheless of some interest.  Theorem \ref{thm:separation} identifies unfrozenness of relations to the structure of Horn clauses defining the class; several examples are given, including Example~\ref{eg:K3=} which shows that the not-necessarily induced subgraphs of direct powers of the complete graph $\mathbb{K}_3$ are those $3$-colourable graphs in which nonadjacent vertices can be distinctly coloured.  Proposition \ref{pro:core} gives a \logspace\ relationship between CSP complexity and the complexity of deciding unfrozenness and frozenness of relations in the case of core structures; the relationship is shown to fail for non-core structures in Example \ref{eg:noncore}.

Part \ref{part:sat} contains the most technical sections of the article.  We build off the work of Gottlob \cite{got} to establish a series of results about robust satisfiability, culminating in the \texttt{NP}-completeness of 2-robust positive 1-in-3SAT (Theorem \ref{thm:SP2}).  
The result also shows that the universal Horn class generated by the $2$-element relational template for positive 1-in-3SAT (which consists of the single ternary relation $\{(1,0,0),(0,1,0),(0,0,1)\}$ on the set $\{0,1\}$) has \texttt{NP}-complete membership problem.  
In fact, as alluded to above, it shows that every class of 1-in-3 satisfiable instances that contains the finite members of this universal Horn class has \texttt{NP}-hard membership problem.  
Along the way, we prove the {\texttt{NP}-completeness of $3$-robust NAE3SAT (Theorem \ref{thm:3robNAE}), and reprove the known \texttt{NP}-completeness of deciding flexible $3$-colourability (Theorem \ref{thm:G3C}).  
This last theorem is proved for graphs in which every edge lies within a triangle, a nontrivial extension of the previous proof in~\cite{AGK}, and we are also able to demonstrate} that the flexibility of colouring extends slightly further to a kind of edge flexibility: for any two edges, we show that aside from a few obvious special cases where the range of colours is limited, we can otherwise  always use  $2$ colours in total amongst the two edges, and also $3$ colours; 
see Proposition \ref{pro:edgerobust}.  The main difficulties of the proof relate to the triangulation of edges, however this is the detail we require for a reduction into positive 1-in-3SAT.

Part \ref{part:varieties} contains applications of the results in Part \ref{part:sat} to hardness of finite membership problems for pseudovarieties.  
In Section \ref{sec:B21} we consider the pseudovariety $\mathsf{V}_{\rm fin}({\bf B}_2^1)$ generated by the six-element Brandt monoid $\mathbf{B}_2^1$ and show in Theorem \ref{thm:LDS} that any pseudovariety containing $\mathsf{V}_{\rm fin}({\bf B}_2^1)$ and contained within the join of $\mathsf{V}_{\rm fin}({\bf B}_2^1)$ with $\mathsf{L}\mathbf{DS}$ has \texttt{NP}-hard finite membership problem (with respect to \logspace\ reductions).  Note that $\mathsf{L}\mathbf{DS}$ is the largest pseudovariety omitting ${\bf B}_2^1$.  
As a corollary we deduce that adding an identity element to any finite completely 0-simple semigroup whose sandwich matrix is block-diagonal produces a semigroup (or monoid) whose pseudovariety has \texttt{NP}-hard finite membership problem; Corollary~\ref{cor:orthodoxlike}.  
In Section~\ref{sec:small} we explore results beyond the monoid and semigroup signatures. We deduce some easier consequences of the flexible satisfaction results by giving a $3$-element algebra and $4$-element groupoid with \texttt{NP}-complete and \texttt{NP}-hard finite membership problem (respectively) for their pseudovarieties; see Corollary~\ref{cor:small}. 
We then consider expansions of ${\bf B}_2^1$ to the inverse semigroup signature and the semiring signature.
\begin{remark}
A number of developments have occurred since the submission of the article, some of which depend on the key results contained here.  
Ham \cite{ham:SATconf,ham:SAT} used Theorem \ref{thm:SP2} below (from arXiv version) to provide a full classification of all constraint problems on 2-elements for which a version of $2$-robust satisfiability is intractable and initiated the algebraic method for both robust satisfiability and universal Horn class membership.  
Following this, the author and Ham revisited the central approach of Part \ref{part:sat} below and showed how to obtain a complete classification of tractability and intractability for robust satisfiability on core templates, as well as of tractability and intractability for universal Horn class membership on core templates \cite{hamjac}.  
In progress work of the author with Barto and Ham will extend this to a complete classification of finite relational structures with tractable universal Horn class membership problem.  
While those developments are in a sense infinitely more powerful than the results obtained in the present article, they do not supersede the results here due to the slightly weakened definitions used in \cite{hamjac}; see discussion after Lemma \ref{lem:folklore} below as well as after Corollary \ref{cor:uHK3}.  
The tighter definitions used in the present results enable computational reductions to instances with a high level of local uniformity, which is very useful in practice and important for the main results of Part \ref{part:varieties}.  

Further developments depending on the present paper can be found in Jackson and Zhang~\cite{jaczha} and Jackson, Ren and Zhao~\cite{JRZ}, which will be discussed in more detail at appropriate points of Part~\ref{part:varieties}.
\end{remark}
\part{Preliminaries: constraints, satisfaction and separation}\label{part:1}

\section{Constraint Satisfaction Problems.}\label{sec:CSPintro}
In this article we consider constraint problems over fixed finite relational templates only.  A \emph{template} $\mathbb{A}$ consists of a finite set $A$ (the \emph{domain}) endowed with a finite family $R_1^\mathbb{A},\dots,R_n^\mathbb{A}$ of relations on $A$, each of some finite arity.  The family of symbols $R_1,\dots,R_n$ and associated arities is the \emph{vocabulary}, or (relational) \emph{signature} of $\mathbb{A}$, so that we can equivalently think of the template $\mathbb{A}$ as a relational structure $\mathbb{A}=\langle A; R_1^\mathbb{A},\dots,R_n^\mathbb{A}\rangle$. 

Two key examples of interest in the present article are the template for graph $3$-colourability, which consists of the simple graph 
\[
\mathbb{K}_3=\langle\{0,1,2\};\neq\rangle
\]
which has a single relation, of arity $2$,
 and the two-element template for what is {commonly called} positive 1-in-3SAT
 \[
 \mathbbold{2}=\langle \{0,1\}; \{(0,0,1),(0,1,0),(1,0,0)\}\rangle
 \]
 which has a single relation, of arity $3$.

The \emph{constraint satisfaction problem} $\CSP(\mathbb{A})$ over $\mathbb{A}$ is the computational problem that takes as an instance a set of variables $x_1,\dots,x_m$ and a set of constrained tuples: expressions of the form $(x_{i_1},\dots,x_{i_k})\in R$, where $R$ is a symbol of arity $k$ from the vocabulary of $\mathbb{A}$.  Such an instance is a YES instance of $\CSP(\mathbb{A})$ if there is an assignment $\phi\colon\{x_1,\dots,x_m\}\to A$ satisfying all of the constraints of the instance: for each constraint $(x_{i_1},\dots,x_{i_k})\in {R}$, we must have $(\phi(x_{i_1}),\dots,\phi(x_{i_k}))\in {R}^\mathbb{A}$.  Otherwise the instance is a NO instance of $\CSP(\mathbb{A})$.
 
An algebraic/model-theoretic perspective reveals that $\CSP(\mathbb{A})$ is nothing other than the  homomorphism problem for $\mathbb{A}$.  To see this, we may consider the set of variables in a CSP instance $I$ as the universe of some relational structure $\mathbb{B}$ of the same signature as $\mathbb{A}$.  For each symbol $R$ in the relational signature of $\mathbb{A}$, the relation $R^\mathbb{B}$ consists of the set of all tuples that were constrained to be $R$-related in $I$:
\[
R^\mathbb{B}:=\{(x_{i_1},\dots,x_{i_k})\mid (x_{i_1},\dots,x_{i_k})\in R\text{ is a constraint in $I$}\}.
\]
Then the notion of ${\phi}\colon\{x_1,\dots,x_m\}\to A$ satisfying the constraints of $I$ coincides with the definition of ${\phi}\colon\mathbb{B}\to\mathbb{A}$ being a homomorphism.  Conversely, any finite relational structure $\mathbb{B}$ of the same type as $\mathbb{A}$ gives a CSP instance by calling the elements of the universe of $\mathbb{B}$ ``variables'' and taking the list of elements of each fundamental relation of $\mathbb{B}$ as the list of constraints.  We often use the notation $\mathbb{B}\in\CSP(\mathbb{A})$ to denote the situation where $\mathbb{B}$ is a YES instance of the computational problem $\CSP(\mathbb{A})$.

\section{Separation and unfrozen relations}\label{sec:sep}
Let $\mathcal{R}$ be a relational signature and $R\in \mathcal{R}\cup\{=\}$.  For a template $\mathbb{A}$ of signature~$\mathcal{R}$, we will say that an $\mathcal{R}$-structure $\mathbb{B}$ satisfies the \emph{$R$-separation condition} with respect to $\mathbb{A}$ if for each $(b_1,\dots,b_k)\notin R^\mathbb{B}$, there is a satisfying assignment $\nu$ (that is, a homomorphism $\nu:\mathbb{B}\to\mathbb{A}$) with $(\nu(b_1),\dots,\nu(b_k))\notin R^\mathbb{A}$, where of course $=$ is to be interpreted as equality on both $\mathbb{A}$ and $\mathbb{B}$.  

Following  Beacham \cite{bea}, we may say that a tuple $(b_1,\dots,b_k)$ is \emph{frozen-in} {to a $k$-ary relation $R$} (with respect to  $\mathbb{A}$) if $(b_1,\dots,b_k)\notin R^\mathbb{B}$ and the {$R$-}separation condition fails at $(b_1,\dots,b_k)$.     We say that the relation $R^\mathbb{B}$ is \emph{unfrozen-in} (with respect to $\mathbb{A}$) if no tuple outside of $R^\mathbb{B}$ is frozen{-}in.  The reader will note that a relation $R$ is unfrozen-in if and only if the $R$-separation condition holds.
As we now see, {this frozenness condition is} closely related to familiar concepts in model theory and universal algebra.  The following definitions can be found Burris and Sankappanavar \cite{bursan} for example; a more detailed treatment is Gorbunov \cite{gor}.  

A \emph{Horn clause} in a first order language (with equality) is a disjunction 
\begin{equation}
\phi_1\vee\phi_2\vee\dots\vee \phi_k\label{eqn:Horn}
\end{equation}
in which each of $\phi_1,\dots,\phi_k$ is either atomic or  negated-atomic, but at most one disjunct $\phi_i$ is not negated.  A first order sentence in this language is a \emph{universal Horn sentence} if it is a universally quantified Horn clause.  (Technically one might take universally quantified conjunctions of Horn clauses, but these are logically equivalent to a finite set of sentences of the form we have described.)

When \emph{exactly one} disjunct is not negated (say, $\phi_k$), the Horn clause \eqref{eqn:Horn} can be written as the implication 
\begin{equation}
(\neg \phi_1\And\dots\And \neg\phi_{k-1})\rightarrow \phi_k\label{eqn:qeqn}
\end{equation}
(of course as $\phi_1,\dots,\phi_{k-1}$ were already negated by assumption, the premise of \eqref{eqn:qeqn} is really just a conjunction of atomic formul{\ae}).  Universally quantified expressions of this form are often called \emph{quasi-equations}, or \emph{quasi-identities}.  
Universal Horn sentences in which \emph{all} disjuncts are negated are often called \emph{anti-identities}, though these are also easily seen to be just the negations of primitive-positive sentences (existentially quantified conjunctions of atomic formul{\ae}). {Atomic formul{\ae} in relational signatures are frequently not equalities, however we follow \cite{gor} and continue to use the words ``equation'' and ``identity'' in these definitions.}

The following theorem is a basic synthesis of a number of commonly encountered conditions.  When $k=0$ (meaning $\{R_1,\dots,R_k\}=\varnothing$), then Theorem \ref{thm:separation} gives the usual syntactic characterisation for the antivariety of a finite structure.  When $\{R_1,\dots,R_k\}$ is $\mathcal{R}\cup\{=\}$, the theorem gives a very widely used series of equivalent conditions for membership in the universal Horn class generated by a finite structure.  When $k=1$ and $R_1$ is the relation $=$, then $\up{({ii})}\Leftrightarrow \up{({iii})}$ of the theorem is essentially Corollary~8 of Stronkowski \cite{str}.  (When $\mathbb{A}$ is infinite, or if we consider an infinite set of structures instead of a single $\mathbb{A}$, then one needs to also account for ultraproduct closure, but there is an obvious extension of statements {(i)} and {(iii)} that holds, with standard adjustments to the proof.)

\begin{thm}\label{thm:separation}
Let $\mathbb{A}$ be a finite relational structure and $\mathbb{B}$ a relational structure of the same finite signature $\mathcal{R}$.  Let $\{R_1,\dots,R_k\}$ be a subset of $\mathcal{R}\cup\{{=}\}$.  The following are equivalent.
\begin{enumerate}
\item[\textup{({i})}] {There is at least one homomorphism from $\mathbb{B}$ into $\mathbb{A}$ and $\mathbb{B}$} satisfies the separation condition for each of $R_1,\dots,R_k$.
\item[\textup{({ii})}] $\mathbb{B}$ satisfies all universal Horn sentences satisfied by $\mathbb{A}$ for which any non-negated disjunct involves relations only from $R_1,\dots,R_k$.
\item[\textup{({iii})}] There is a homomorphism from $\mathbb{B}$ into a  nonzero direct power of $\mathbb{A}$ that preserves $\{\neg R_1,\dots,\neg R_k\}$.
\end{enumerate}
\end{thm}
\begin{proof}
({i})$\Rightarrow$({iii}).  
Assume ({i}) and let $\hom(\mathbb{B},\mathbb{A})$ be the family of all homomorphisms from $\mathbb{B}$ into $\mathbb{A}$; this is nonempty because there is at least one homomorphism from $\mathbb{B}$ to $\mathbb{A}$.  
There is a natural  homomorphism from $\mathbb{B}$ into $\mathbb{A}^{\hom(\mathbb{B},\mathbb{A})}$ given by identifying each $b$ with the function $e_b\colon\hom(\mathbb{B},\mathbb{A})\to \mathbb{A}$ defined by $e_b({\nu})={\nu}(b)${, for each $\nu\in\hom(\mathbb{B},\mathbb{A})$}.  
This is a homomorphism, but also it preserves $\{\neg R_1,\dots,\neg R_k\}$: for each $R\in \{R_1,\dots,R_k\}$ and $(b_1,\dots,b_\ell)\notin R^\mathbb{B}$, the separation condition of ({i}) ensures that there is ${\nu}\colon\mathbb{B}\to\mathbb{A}$ with $({\nu}(b_1),\dots,{\nu}(b_\ell))\notin R{^\mathbb{A}}$, so that $(e_{b_1},\dots,e_{b_\ell})({\nu})=({\nu}(b_1),\dots,{\nu}(b_\ell))\notin R^{\mathbb{A}}$.  Thus ({iii}) holds.

({iii})$\Rightarrow$({i}).  Assume ({iii}) holds.   Let $\mathbb{B}'$ be the image of $\mathbb{B}$ under the assumed homomorphism.  Then whenever $(b_1,\dots,b_\ell)\notin R^\mathbb{B}$ for some $R\in \{R_1,\dots,R_k\}$ we also have $(b_1,\dots,b_\ell)\notin R^{\mathbb{B}'}$, and projecting onto a coordinate $i$ in which $(b_1(i),\dots,b_\ell(i))\notin R^{\mathbb{A}}$ gives the desired homomorphism from $\mathbb{B}$ into $\mathbb{A}$ witnessing the $R$-separation condition at the (arbitrary) tuple $(b_1,\dots,b_\ell)\notin R^\mathbb{B}$.  This shows that ({i}) holds.

({ii})$\Rightarrow$({i}).  {We first prove this in the case that $\mathbb{B}$ is finite.}  Assume that ({i}) fails.   Let $\operatorname{diag}^+(\mathbb{B})$ denote positive diagram of $\mathbb{B}${: the conjunction of all atomic formul{\ae} in elements $B$ that are true in $\mathbb{B}$}.  
{The negation $\neg\operatorname{diag}^+(\mathbb{B})$ can be universally quantified to produce a sentence that is logically equivalent to an anti-identity.}  
If {there is no homomorphism from $\mathbb{B}$ to $\mathbb{A}$}, then $\mathbb{A}$ satisfies {(the universal quantification of)} $\neg\operatorname{diag}^+(\mathbb{B})$, {an anti-identity} which trivially fails on $\mathbb{B}$.  

Now assume that there is a homomorphism from $\mathbb{B}$ to $\mathbb{A}$.   
It follows  that for some $R\in \{R_1,\dots,R_k\}$, the $R$-separation condition fails at a tuple $(b_1,\dots,b_\ell)\in B^\ell\backslash R^\mathbb{B}$.  
Then $\mathbb{A}$ satisfies the quasi-identity $\operatorname{diag}^+(\mathbb{B})\rightarrow (b_1,\dots,b_\ell)\in R$ {(again implicitly universally quantified)}, while again $\mathbb{B}$ trivially fails this.  
Hence ({ii}) fails.  {Now we consider the case when $\mathbb{B}$ is infinite.  
We prove $\up{(ii)}\Rightarrow\up{(i)}$ using the fact that $\mathbb{B}$ embeds into an ultraproduct of its finite substructures.  Assume that $\mathbb{B}$ satisfies {(ii)}.  Then~(ii)  also holds on all finite substructures of $\mathbb{B}$.  Hence condition~(i) also holds on all finite substructures of $\mathbb{B}$.  
Basic model theory arguments then easily extend~(i) to any ultraproduct of these finite structures, thereby establishing~(i) for~$\mathbb{B}$.  We give a brief sketch of the details for those unfamiliar with such arguments.  
Let $I$ be the family of finite subsets of the universe $B$ and for each $i\in I$, let $\mathbb{B}_i$ denote the substructure of~$\mathbb{B}$ on $i$.  
Let $\mathscr{U}$ be an ultrafilter of the Boolean algebra of all subsets of $I$.  
Consider any set of homomorphisms $\{\phi_i\colon\mathbb{B}_i\to\mathbb{A}\mid{i\in I}\}$, one for each $\mathbb{B}_i$.
For each element $\underline{b}/\mathscr{U}$ of the ultraproduct $\prod_{i\in I}\mathbb{B}_i/\mathscr{U}$, the finiteness of $A$ ensures that there is a unique element $a$ of $A$ for which the set $\{i\in I\mid \phi_i(\underline{b}(i))=a\}$ is a member of $\mathscr{U}$.  We denote this element (which depends on $\underline{b}/\mathscr{U}$) by $\phi(\underline{b}/\mathscr{U})$; this function $\phi$ is a homomorphism $\phi\colon\prod_{i\in I}\mathbb{B}_i/\mathscr{U}\to \mathbb{A}$ so that at least one such homomorphism exists.  
If the $R$-separation condition holds for some relation~$R$ on all of the $\mathbb{B}_i$, then for any tuple $(\underline{b}_1/\mathscr{U},\dots,\underline{b}_n/\mathscr{U})$ that falls outside of $R$ in the ultraproduct $\prod_{i\in I}\mathbb{B}_i/\mathscr{U}$, we may select a set $\{\phi_i\colon\mathbb{B}_i\to\mathbb{A}\mid i\in I\}$ such that 
\[
\{i\in I\mid (\phi_i(\underline{b}_1(i)),\dots,\phi_i(\underline{b}_j(i)))\notin R^{\mathbb{B}_i}\}\in \mathscr{U}.
\]  
The homomorphism $\phi:\prod_{i\in I}\mathbb{B}_i/\mathscr{U}\to\mathbb{A}$ determined by this set then witnesses the $R$-separation condition for  $(\underline{b}_1/\mathscr{U},\dots,\underline{b}_n/\mathscr{U})$ with respect to $\mathbb{A}$.
}

({i})$\Rightarrow$({ii}).  Finally, assume that ({ii}) fails for some universal Horn sentence $\psi$.  Let $x_1,\dots,x_n$ be the variables appearing in $\psi$, and let $b_1,\dots,b_n$ be elements of~$B$ such that $\mathbb{B}$ fails $\psi$ at the substitution $x_i\mapsto b_i$.  
If $\psi$ is an anti-identity, then there is no homomorphism from $\mathbb{B}$ into $\mathbb{A}$ {because the substitution $x_i\mapsto b_i$ can be followed by such a homomorphism $\nu\colon\mathbb{B}\to\mathbb{A}$ to yield a contradictory failure of the anti-identity $\psi$ in $\mathbb{A}$}.  
Now assume that one disjunct in $\psi$ is an atomic formula $(x_{i_1},\dots,x_{i_\ell})\in R$ (with $R$ in $\{R_1,\dots,R_k\}$), so that $\psi$ can be written as a quasi-identity, with conclusion $(x_{i_1},\dots,x_{i_\ell})\in R$.  So, as $\psi$ fails at $b_1,\dots,b_n$, we have that $(b_{i_1},\dots,b_{i_\ell})\notin R^\mathbb{B}$.  {We show that the $R$-separation condition fails here.}  Let ${\nu}\colon\mathbb{B}\to\mathbb{A}$ be an arbitrary homomorphism.  
Then ${\nu}$ yields a satisfying interpretation of the premise of $\psi$ in $\mathbb{A}$ by giving each $x_i$ the value ${\nu}(b_i)$.  As $\mathbb{A}\models \psi$, we have that $({\nu}(b_{i_1}),\dots,{\nu}(b_{i_\ell}))\in R^\mathbb{A}$.   Hence the $R$-separation condition fails, as required. 
\end{proof}
When $\{R_1,\dots,R_k\}$ is $\mathcal{R}\cup\{=\}$ in Theorem \ref{thm:separation}, part ({iii}) states that $\mathbb{B}$ is isomorphic to an induced substructure of a direct power of $\mathbb{A}$.  The class of induced substructures of finite direct powers of $\mathbb{A}$ is denoted by $\mathsf{SP}(\mathbb{A})$, where $\mathsf{S}$ is the class operator returning isomorphic copies of induced substructures, and $\mathsf{P}$ is the class operator returning nonempty-index set direct powers of $\mathbb{A}$.  {We use $\mathsf{P}_{\rm fin}$ to denote finite direct powers.}
{When $k=1$ and $R_1$ is the  relation of equality, item~(iii) states that there is an injective homomorphism into a direct power of $\mathbb{A}$.  We give an explicit example of this in the case of graphs, which we here consider as relational structures in the signature of a single binary relation that is symmetric (in other words, directed graphs with symmetric edge relation).  The one-element looped graph will be referred to as the \emph{trivial graph}.}
\begin{eg}\label{eg:K3=}
The class of graphs arising a{s}  \up(not necessarily induced\up) subgraphs of  direct powers of the complete graph $\mathbb{K}_3$ {coincides with} the class of {nontrivial graphs} satisfying the separation condition for $=$.  Equivalently, an \up(undirected\up) graph is {isomorphic to} a subgraph of a nonempty direct power of $\mathbb{K}_3$ if and only if it is nontrivial, and every pair of {distinct}  vertices can be coloured distinctly.
\end{eg}
\begin{proof}
These statements are simply $\textup{(i)}\Leftrightarrow\textup{(iii)}$ of Theorem \ref{thm:separation} applied to $\mathbb{A}:=\mathbb{K}_3$ and with $k=1$ and  $R_1$ chosen to be $=$.
The input structure $\mathbb{B}$ is restricted to be a graph, so that an injective homomorphism from $\mathbb{B}$ into a nonempty direct power of $\mathbb{K}_3$ means that $\mathbb{B}$ is isomorphic to a (not necessarily induced) subgraph of a nonempty direct power of $\mathbb{K}_3$.  Note that the separation condition for $=$ with respect to $\mathbb{K}_3$ is exactly the property that distinct vertices can be distinctly coloured. 
\end{proof} 
{The following example provides} structures achieving the ``defrosting'' of relations in the graph $\mathbb{K}_3$.  
\begin{figure}
\begin{tikzpicture}
\node at (0,0) {
\begin{tikzpicture}
\node (0) [draw,thick,circle,inner sep=1mm, fill = white] at (0,0) {};
\node (1) [draw,thick,circle,inner sep=1mm, fill = white] at (1,0) {};
\node (2) [draw,thick,circle,inner sep=1mm, fill = white] at (2,0) {};
\node (3) [draw,thick,circle,inner sep=1mm, fill = white] at (.5,1) {};
\node (4) [draw,thick,circle,inner sep=1mm, fill = white] at (1.5,1) {};

\draw (0) -- (1) -- (2) -- (4) -- (3) -- (1) -- (4);
\draw (0) -- (3);
\node at (1,-.5) {$\mathbb{C}_3$};
\end{tikzpicture}};
\node at (3,0) {\begin{tikzpicture}
\node (0) [draw,thick,circle,inner sep=1mm, fill = white] at (0,0) {};
\node (1) [draw,thick,circle,inner sep=1mm, fill = white] at (1,0) {};
\node (2) [draw,thick,circle,inner sep=1mm, fill = white] at (.5,1) {};
\node (3) [draw,thick,circle,inner sep=1mm, fill = white] at (1.5,1) {};

\draw (0) -- (1) -- (2) -- (3) -- (1);
\draw (0) -- (2);
\node at (0.75,-.5) {$\mathbb{D}_3$};
\end{tikzpicture}};

\node at (5.5,0) {\begin{tikzpicture}
\node (0) [draw,thick,circle,inner sep=1mm, fill = white] at (0,0) {};
\node (1) [draw,thick,circle,inner sep=1mm, fill = white] at (1,0) {};
\node (2) [draw,thick,circle,inner sep=1mm, fill = white] at (.5,1) {};

\draw (0) -- (1) -- (2) -- (0);
\node at (0.5,-.5) {$\mathbb{K}_3$};\end{tikzpicture}};

\end{tikzpicture}
\caption{Three finite graphs generating universal Horn classes with various stages of unfrozenness relative to homomorphisms into~$\mathbb{K}_3$; see Example \ref{eg:K3}.}\label{fig:K3}
\end{figure}
\begin{eg}\label{eg:K3}
The examples concern the three graphs shown in Figure \ref{fig:K3}.
\begin{enumerate}
\item[{\textup{(i)}}] $3$-colourability.  The class of $3$-colourable graphs coincides with class of \up(isomorphic copies of\up) induced subgraphs of direct powers of the graph $\mathbb{C}_3$.    
\item[{\up{(ii)}}] Edge-unfrozen-in $3$-colourability.  The class of  graphs that {have no loops and} satisfy,  relative to $\mathbb{K}_3$, the separation property for their edge relation coincides with the class of isomorphic copies of induced subgraphs of direct powers of the graph $\mathbb{D}_3$.  
\item[{\up{(iii)}}] Edge- and equality-unfrozen-in $3$-colourability.  The class of graphs that {are nontrivial and} satisfy, relative to $\mathbb{K}_3$, the separation property for both their edge relation and equality coincides with the class of isomorphic copies of induced subgraphs of direct powers of $\mathbb{K}_3$.
\end{enumerate}
\end{eg}
\begin{proof}
Part ({i}) is due to Ne\v{s}et\v{r}il and Pultr \cite{nespul}, though there is a similar construction also in Wheeler \cite{whe}.  Part ({iii}) can be found in Trotta \cite[Lemma 4.1]{tro}.  There are obvious extensions to $k$-colourability given in each of these articles.  For part~({ii}), {let the vertices of {$\mathbb{D}_3$} be numbered clockwise $0,1,2,3$ in the drawing of Figure \ref{fig:K3} so that vertices 0 and 2 are of degree $2$}.  Now let $\mathbb{G}$ be a graph whose edge relation is unfrozen-in (relative to $\mathbb{K}_3$) {and which has no loops}.   Then the edge relation remains unfrozen-in with respect to ${\mathbb{D}_3}$ because $\mathbb{K}_3$ is an induced subgraph of ${\mathbb{D}_3}$.  We claim that equality {on $\mathbb{G}$} is also unfrozen-in with respect to ${\mathbb{D}_3}$.  {To prove this, let} $u\neq v$ {be} distinct vertices of $\mathbb{G}$.  {Because there are no loops, there is at least one non-adjacency, so the edge-separation condition implies that there is at least one $3$-colouring of $\mathbb{G}$.  Let $\nu\colon\mathbb{G}\to \mathbb{K}_3$ be such a $3$-colouring.}  If $\nu(u)\neq\nu(v)$ we are done because $\mathbb{K}_3$ is an induced subgraph of ${\mathbb{D}_3}$.  Otherwise, if $\nu(u)=\nu(v)$, then we can assume without loss of generality that $\nu(u)=\nu(v)=0$.  
Define a new map from $\mathbb{G}$ into ${\mathbb{D}_3}$ by $\nu'(w):=\nu(w)$ for all vertices $w$, except for $v$ where we define $\nu'(v)=3$.  This is a homomorphism {satisfying} $\nu'(u)\neq\nu'(v)$.  So, by Theorem \ref{thm:separation} it follows that $\mathbb{G}$ is isomorphic to an induced subgraph of a power of ${\mathbb{D}_3}$.

{For the converse, let $\mathbb{G}$ be an induced subgraph of a power of $\mathbb{D}_3$.  Note that the induced subgraph of $\mathbb{D}_3$ on $\{0,1,2\}$ is isomorphic to $\mathbb{K}_3$ and that $\mathbb{D}$ retracts onto this subgraph.  As $\mathbb{G}\in\mathsf{SP}(\mathbb{D}_3)$ we may follow any projection from $\mathbb{G}$ with this retraction to obtain a homomorphism from $\mathbb{G}$ to $\mathbb{K}_3$.  To show that the edge relation of $\mathbb{G}$ is unfrozen-in with respect to $\mathbb{K}_3$,} let $u$ and $v$ be such that $\{u,v\}$ is not an edge of $\mathbb{G}$.  Then there is a projection $\nu:\mathbb{G}\to {\mathbb{D}_3}$ with $\{\nu(u),\nu(v)\}$ not an edge of ${\mathbb{D}_3}$.  Follow this by the retraction of ${\mathbb{D}_3}$ onto {the induced subgraph $\mathbb{K}_3$} on vertices $\{0,1,2\}$.  This retraction maps nonedges of ${\mathbb{D}_3}$ to nonedges of $\mathbb{K}_3$.  Thus we have a $3$-colouring of $\mathbb{G}$ in which the nonedge $\{u,v\}$ is mapped to a nonedge of~$\mathbb{K}_3$.
\end{proof} 
Deciding membership in the class of  $3$-colourable graphs with the {$=$-separation condition with respect to $\mathbb{K}_3$} was shown to be \texttt{NP}-complete by Beacham \cite{bea}, while Ham \cite{ham:SAT} used Theorem \ref{thm:1in3} of the present article to classify the \texttt{NP}-completeness of deciding the $=$-separation condition in the case of Boolean templates.  Deciding membership in the class of graphs {simultaneously satisfying both the edge and equality separation conditions with respect to $\mathbb{K}_3$} was shown to be \texttt{NP}-complete by Abramsky, Gottlob and Kolaitis~\cite{AGK}.  This property coincides with ``$2$-robust $3$-colourability'', which we {will describe in Section \ref{sec:robustsat} below.}

\begin{remark}
 {Case (i) of Example \ref{eg:K3} shows that $\CSP(\mathbb{K}_3)=\mathsf{SP}_{\rm fin}(\mathbb{C}_3)$.}  In Jackson and Trotta~\cite[Lemma~6]{jactro} it is shown that for every finite template ${\bf T}$ of finite signature, there is a finite template ${\bf T}^\sharp$ of the same signature and with 
$\CSP({\bf T})=\CSP({\bf T}^\sharp)=\mathsf{SP}_{\rm fin}({\bf T}^\sharp)$.  In other words, there is always an equivalent template ${\bf T}^\sharp$ relative to which all YES instances of $\CSP({\bf T})=\CSP({\bf T}^\sharp)$ have every relation (and equality) unfrozen-in. 
\end{remark}

\section{Robust satisfaction}\label{sec:robustsat}
The various separation and identification conditions of Section \ref{sec:sep} have a close relationship with ``robust satisfiability'' in the sense of Abramsky, Gottlob and Kolaitis \cite{AGK}.  A \emph{partial assignment} from $\mathbb{B}$ to $\mathbb{A}$ is a map~$\phi$ from some subset $S$ of the underlying universe~$B$ into~$\mathbb{A}$.  The notion of robust satisfiability concerns the ability to extend  partial assignments to satisfying {total} assignment{s}.  For this it is necessary that~$\phi$ satisfy any existing constraints inherited by the induced substructure of $\mathbb{B}$ on $S$, but also any additionally implicit constraints that might be imposed on $S$ by the larger structure $\mathbb{B}$.  In Abramsky, Gottlob and Kolaitis \cite{AGK}, a partial assignment is said to be  \emph{locally compatible} if it preserves all projections of  constraints, in addition to the constraints themselves.
Then the structure $\mathbb{B}$ is \emph{$k$-robustly satisfiable} if every locally compatible partial assignment from {every} $k$-element subset of $B$ into $\mathbb{A}$ extends to a satisfying assignment of $\mathbb{B}$ in $\mathbb{A}$.  {Note that $(k+1)$-robust satisfiability does not in general imply $k$-robust satisfiability.  We write $\leq k$-robust satisfiability to mean $\ell$-robust satisfiability for each $\ell\leq k$; see \cite[Lemma 3]{hamjac} for an alternative approach.}

{As an example, consider the template $\mathbbold{2}$ for positive 1-in-3SAT introduced in Section \ref{sec:CSPintro};  the domain is $\{0,1\}$ and the single ternary relation $R$ which in $\mathbbold{2}$ is given by $R^\mathbbold{2}:=\{(1,0,0), (0,1,0),(0,0,1)\}$.   Let $\pi_{1,2}(R)$ denote the projection of $R$ to the first two coordinates: so $\pi_{1,2}(R^\mathbbold{2})=\{(0,0),(0,1),(1,0)\}$ (there are two other such binary projections, but we focus here on $\pi_{1,2}(R)$ for simplicity).  If in an instance~$\mathbb{B}$ of 1-in-3SAT there are two elements $x,y\in B$ that lie in $\pi_{1,2}(R^\mathbb{B})$ (meaning there is some $z$ such that $(x,y,z)\in R^{\mathbb{B}}$)} then a partial truth assignment would fail to be {locally} compatible if it assigned both $x$ and $y$ the value $1$, as this lies outside of $\pi_{1,2}(R^\mathbbold{2})$.

There are natural variations of this notion of local compatibility.  Recall that a relation is \emph{primitive-positive definable} (abbreviated to \emph{pp-definable}) from a set of relations $\{R_1,\dots,R_k\}$ if it is the solution set to an existentially quantified formula built from the relations $R_1,\dots,R_k$ and conjunction.  It is easily verified that pp-definable relations must be preserved by any homomorphism preserving $\{R_1,\dots,R_k\}$.  The notion of local compatibility is an example of a pp-definition, as the projection of a $k$-ary relation $R$ to some subset  $\{i_1,\dots,i_\ell\}\subseteq \{1,\dots,k\}$ is the relation $\{(x_{i_1},\dots,x_{i_\ell})\mid \exists x_{j_1}\dots \exists x_{j_{k-\ell}}\ (x_1,\dots,x_k)\in R\}$, where $i_1<\dots<i_\ell$, $j_1<\dots<j_{k-\ell}$ and $\{j_1,\dots,j_{k-\ell}\}=\{1,\dots,k\}\backslash\{i_1,\dots,i_\ell\}$.  The following lemma is probably folklore.
\begin{lem}\label{lem:folklore}
Let $\mathbb{B}$ be an instance of $\CSP(\mathbb{A})$, and $\nu:S\to A$ be a partial assignment from some subset $S$ of $B$.  Then $\nu$ extends to a homomorphism if and only if $\nu$ preserves the restriction to $S$ of all {pp-definable relations} from the fundamental relations of $\mathbb{B}$.
\end{lem}
\begin{proof}
Certainly every homomorphism preserves primitive-positive definable relations.  So it suffices to show that if $\nu$ preserves every pp-definable relation then it extends to a homomorphism.
Let $s_1,\dots,s_\ell$ be an enumeration of the elements of $S$.  Now let $\delta(s_1,\dots,s_\ell)$ denote the {positive diagram $\operatorname{diag}^+(\mathbb{B})$ (so that} each element of $B$ is a variable) but with all variables formed from $B\backslash S$ existentially quantified and the variables $s_1,\dots,s_\ell$ formed from $S$ left unquantified.  Let $R$ denote the $\ell$-ary relation that is pp-defined by $\delta(s_1,\dots,s_\ell)$; obviously $(s_1,\dots,s_\ell)$ is a tuple in $S\cap R$.  Assume then that $\nu$ preserves the restriction of $R$ to $S$.  So in particular, $(\nu(s_1),\dots,\nu(s_\ell))\in R^\mathbb{A}$.  Now each $b\in B\backslash S$ appears as existentially quantified in~$\delta$, and hence there is an evaluation of each $b\in B\backslash S$ in $A$ such that the {positive} diagram of $\mathbb{B}$ holds.  This means there is a homomorphism from $\mathbb{B}$ into $\mathbb{A}$ that agrees with $\nu$ on $S$.
\end{proof}
Thus, requiring that \emph{all} pp-definable relations be preserved is the same as requiring that $\nu$ extend to a homomorphism, which is hardly local.  By restricting the number of quantified variables in a pp-formula though, we obtain a natural hierarchy of local compatibility conditions.  We could say for example that a relation is \emph{$k$-local} if it can be defined from the fundamental relations by way of a pp-formula with at most $k$ existentially quantified variables.  

Consider for example the relational template for positive NAE3SAT, which consists of the domain $\{0,1\}$ with the one ternary relation $R$ given on $\{0,1\}$ by $\{(0,0,1),(0,1,0),(1,0,0),(1,1,0),(1,0,1),(0,1,1)\}$.  Consider the instance on variables $v,w,x,y,z$ with $(v,w,x)\in R$ and $(x,y,z)\in R$.  The partial assignment $\nu:\{v,w,y,z\}\to \{0,1\}$ given by $\nu(v)=\nu(w)=1$ and $\nu(y)=\nu(z)=0$ preserves all projections of the relation $R$, however it cannot extend to a satisfying truth assignment because it fails the pp-definable $4$-ary relation $\{(a,b,d,e)\mid \exists c\ (a,b,c)\in R\And (c,d,e)\in R\}$, under which $(v,w,y,z)$ is related and which on $\{0,1\}$ avoids the tuple $(1,1,0,0)$.    This example demonstrates that it is almost impossible for a positive NAE3SAT instance to be $4$-robustly satisfiable: it would require that all clauses either do not overlap at all or overlap by two variables.  However, asking for preservation of $1$-local relations might allow for a much more powerful concept of $4$-robust NAE-satisfiability in NAE3SAT.  {These suggestions were taken to full generalisation by Ham \cite{ham:SATconf,ham:SAT}, where \texttt{NP}-completeness with respect to some finite level of locality was characterised for Boolean constraint problems.  The subsequent article~\cite{hamjac} extended this to other finite relational templates.}

{For any finite relational structure $\mathbb{A}$, let $\mathbb{A}_{\rm const}$ denote the result of adjoining all singleton unary relations to the signature of $\mathbb{A}$.}  {Now} recall that a relational structure is a \emph{core} if every endomorphism is an automorphism.  Every finite relational structure $\mathbb{A}$ retracts onto a core substructure $\mathbb{A}^\flat$, so that $\CSP(\mathbb{A})$ is identical to $\CSP(\mathbb{A}^\flat)$.  {For any core structure $\mathbb{A}$, the problem $\CSP(\mathbb{A})$ is logspace equivalent to $\CSP(\mathbb{A}_{\rm const})$}  \cite{BJK,lartes}.  {This observation implies that in the following proposition, when $\mathbb{A}$ itself is a core, then $\CSP(\mathbb{A}_{\rm const})$ can be replaced by $\CSP(\mathbb{A})$.}
\begin{pro}\label{pro:core}
Let $\mathbb{A}=\langle A;R_1^{{\mathbb{A}}},\dots,R_{{\ell}}^{{\mathbb{A}}}\rangle$ be a finite relational structure of finite signature.   There is a logspace Turing reduction from each of the following decision problems  to $\CSP(\mathbb{A}_{\rm const})$\up:
\begin{enumerate}
\item[\textup{(i)}] deciding ${\leq} k$-robust satisfiability of a structure $\mathbb{B}$ relative to $\mathbb{A}$\up;
\item[\textup{(ii)}] deciding of a given structure $\mathbb{B}$ if the separability condition holds for any fixed subset of relations from $\{R_1,\dots,R_{{\ell}}\}\cup\{=\}$.
\end{enumerate}
\end{pro}
\begin{proof}
We  make use of the singleton unary relations that appear in $\mathbb{A}_{\rm const}$: for each $a\in A$, let $c_a$ be the unary relation defined on $A$ by $c_a^\mathbb{A}=\{a\}$.

{We first describe a preliminary ``preprocessing'' construction.  Let $\mathbb{B}$ be an instance of one of the above problems and let $m$ be any fixed positive integer; we will use the case of $m=k$ for (i) of the proof and for part (ii) we choose $m$ to be the maximal arity of the fundamental relations.  For any partial assignment $\phi$ from a subset of $B$ of size at most $m$ to $A$ we may construct a new structure $\mathbb{B}_\phi$ in the expanded signature of $\mathbb{A}_{\rm const}$ from $\mathbb{B}$ as follows: we simply output $\mathbb{B}$, but with each of the new relation symbols $\{c_a\mid a\in A\}$ defined by  $c_{a}^{\mathbb{B}}:=\phi^{-1}(a)$.  
This can be constructed in logspace and has the property that $\phi$ extends to a homomorphism if and only if $\mathbb{B}$ is a YES instance of $\CSP(\mathbb{A}_{\rm const})$.}
  
{To decide $k$-robustness, we will successively test each locally compatible partial assignment on every $k$ points as follows. Given a partial assignment $\phi$ on $k$ elements of $B$, first test that $\phi$ is locally compatible (because $k$ is a fixed number this can be done in logspace).  If not, discard and construct the next possible partial assignment.  If $\phi$ is locally compatible, then write the structure $\mathbb{B}_\phi$ to the query tape for a $\CSP(\mathbb{A}_{\rm const})$-oracle to verify that $\phi$ extends to a homomorphism.  If it does not, then reject.  If $\phi$ does extend, then continue the verification of $k$-robustness by trialling the next partial assignment or tuple.}
{Note that the oracle tape is not part of the workspace, the query string does not need to be logarithmic in length and the oracle tape can be assumed to be automatically cleaned at no cost after each query (see \cite[Section 6.2]{BDJN} for example).}

{Deciding separation is similar.  Our algorithm will consult a dictionary consisting of the complements of each fundamental relation on $\mathbb{A}$: for each fundamental relation $R$ or arity $k$, construct $S:=A^k\backslash R^\mathbb{A}$.  These are fixed, so not part of the logspace reduction. The $R$-separation property for a tuple $(b_1,\dots,b_{{n}})\notin R^\mathbb{B}$ (of arity {$n$}) can be decided as follows.  Successively (using the same section of work tape each time), consider each $(a_1,\dots,a_{{n}})\in S^\mathbb{A}$, and for each, let $\phi$ be the partial assignment mapping $b_i$ to $a_i$.  For each such tuple, query $\mathbb{B}_\phi$ to $\CSP(\mathbb{A}_{\rm const})$ to discover if $\phi$ extends to a homomorphism.  If all tuples $(a_1,\dots,a_{{n}})\in S^\mathbb{A}$ return failure then  $R$-separation fails at $(b_1,\dots,b_{{n}})\notin R^\mathbb{B}$.  Otherwise a tuple is found that returns success and the $R$-separation is verified for the tuple $(b_1,\dots,b_{{n}})\notin R^\mathbb{B}$.  This process can be performed on the same segment of work tape for each  tuple $(b_1,\dots,b_k)\notin R^\mathbb{B}$.}
\end{proof}
{A special case of (ii) in Proposition \ref{pro:core} is that there is a reduction from membership in $\mathsf{SP}(\mathbb{A})$ to $\CSP(\mathbb{A}_{\rm const})$, and even to $\CSP(\mathbb{A})$ in the case that $\mathbb{A}$ is a core.}
\begin{eg}\label{eg:noncore}
Recall that the complete graph $\mathbb{K}_2$ is the usual template for {the graph 2-colourability problem} G2C \up(identically, positive 1-in-2SAT\up).  As $\mathbb{K}_2$ is a core structure and $\CSP(\mathbb{K}_2)$ is decidable in logspace, then membership of finite structures in $\mathsf{SP}(\mathbb{K}_2)$ is also decidable in logspace, as is $k$-robust $2$-colourability \up(for any fixed $k$\up).  Similarly, if {$\mathbb{A}$} denotes the template {$\langle \{0,1\};\leq,\{0\},\{1\}\rangle$} for directed graph unreachability {\up(see Immerman \cite{imm} and Larose and Tesson \cite[Theorem 4.1]{lartes}\up),} then $\mathsf{SP}({\mathbb{A}})$ has membership problem decidable in nondeterministic logspace and $k$-robust unreachability is decidable in nondeterministic logspace.
\end{eg}

The assumption of being a core in Proposition \ref{pro:core} is necessary, at least when it comes to testing the separation conditions.  
{\begin{eg} 
Let $\mathbb{C}_3^+$ be the result of adjoining a single looped vertex to the graph~$\mathbb{C}_3$ of Example \ref{eg:K3}.  Then $\CSP(\mathbb{C}_3^+)$ is trivial but deciding membership in the universal Horn class of $\mathbb{C}_3^+$ is \texttt{NP}-complete. 
\end{eg}
\begin{proof}
The loop trivially ensures that $\CSP(\mathbb{C}_3^+)$ is trivial (all graphs are YES instances).  Now let $\mathbb{G}$ be any finite connected graph with more than one vertex.  If $\mathbb{G}$ is $3$-colourable then it lies in $\mathsf{SP}(\mathbb{C}_3)\subseteq \mathsf{SP}(\mathbb{C}_3^+)$.  Now assume that $\mathbb{G}$ lies in $\mathsf{SP}(\mathbb{C}_3^+)$.  As $\mathbb{G}$ is nontrivial, there is at least one tuple that requires a witnessing homomorphism for $=$-separation.  Because $\mathbb{G}$ is connected, this requires a homomorphism from $\mathbb{G}$ into the subgraph $\mathbb{C}_3$.  Thus $\mathbb{G}$ is $3$-colourable.  
\end{proof}}

\part{Flexible satisfaction}\label{part:sat}
{In this part we develop a series of reductions culminating in the \texttt{NP}-completeness of ${\leq}2$-robust satisfiability for positive 1-in-3SAT,  with some extra technical properties that we make use of in Part \ref{part:varieties}.  Along the way, we also deduce the \texttt{NP}-completeness of ${\leq}2$-robust graph 3-colourability and the \texttt{NP}-completeness of ${\leq}3$-robust satisfiability for NAE3SAT.}

{The technical properties we eventually require in Part \ref{part:varieties} are accumulated through each of the reductions that we perform in Part \ref{part:sat}, which we account for by incrementing the index starting from some original instance $I$ (of the problem positive NAE3SAT), as $I=I_0$, $I_1$, $I_2$,\dots.}
\section{Robust NAE3SAT}\label{sec:sat}
The main result in this section is Theorem  \ref{thm:3robNAE} demonstrating the \texttt{NP}-completeness of ${\leq}3$-robust {NAE}3SAT, which is obtained by reduction from positive NAE3SAT via the intermediate problem of positive NAE$(3k+3)$SAT.  The arguments are parallel to those of Gottlob~\cite{got} and Ham and Jackson \cite[\S10]{hamjac} (the latter proving a result sketched in Abramsky, Gottlob and Kolaitis~\cite{AGK}) for conventional satisfiability (as opposed to NAE-satisfiability), though the details are considerably different for the second reduction.  

We commence  with an instance $I=I_0$ of the positive NAE3SAT problem and construct an instance $I_1$ of positive NAE$(3k+3)$SAT (we later choose $k=6$).  This reduction will be a trivial variation of that of Gottlob~\cite{got}.

We let the variables of $I$ be denoted by $X=\{x_1,\dots,x_n\}$ and for any~$k$, consider the set 
\[
X_k=\{x_{i,j}\mid i\in\{1,\dots,n\} \text{ and }j\in\{1,\dots,2k+1\}\}.
\]   
Following Gottlob~\cite{got}, we construct an instance $I_1$ of positive NAE$(3k+3)$SAT as follows: each clause $(x\vee y\vee z)$ in $I$ (so, $x,y,z\in X$) is replaced by the family of all $\binom{2k+1}{k+1}^3$ clauses of the form
\[
(x_{i_1}\vee\dots \vee x_{i_{k+1}}\vee y_{i_1'}\vee\dots \vee y_{i_{k+1}'}\vee z_{i_1''}\vee\dots \vee z_{i_{k+1}''})
\]
where $\{i_1,\dots,i_{k+1}\}$, $\{i_1',\dots,i_{k+1}'\}$, $\{i_1'',\dots,i_{k_{k+1}}''\}$ are $(k+1)$-element subsets of $\{1,\dots,2k+1\}$.  Here if $x=x_i$ then by $x_{i_j}$ we mean $x_{i,i_j}$ and similarly for $y,z$.  
 \begin{thm}\label{thm:NAE3k3}
{For any $k>0$, the problem positive NAE$(3k+3)$SAT is \texttt{NP}-complete.  Moreover, the instance $I_1$ is either a NO instance \up(when $I$ is a NO instance\up), or has the property  that every partial assignment on at most $k$ variables extends to a NAE-satisfying truth assignment.}
\end{thm}
\begin{proof}  This proof is really just a routine verification that Gottlob's argument in~\cite{got} for reducing 3SAT to $(3k+3)$SAT continues to hold in the NAE setting, and is included because this verification needs to be done, and for relative completeness of the rest of the proof, which depends on this initial step.

We show that the mapping $I_0\mapsto I_1$ is a  reduction that takes NO instances of positive NAE3SAT to NO instances of positive NAE$(3k+3)$SAT and YES instances of positive NAE3SAT to $\leq k$-robustly satisfiable instances of positive NAE$(3k+3)$SAT.  It is easy to see that this mapping can be performed in linear time, albeit with a fairly large constant, and in logspace; moreover no negations are added, so the construction trivially preserves positive instances.

Each NAE-satisfying truth assignment $\nu$ for $I$ extends to a NAE-satisfying truth assignment $\nu^\sharp$ for $I_1$ by giving each $x_{i,j}$ the value given $x_i$ by $\nu$.  So $I$ is a YES instance implies $I_1$ is a YES instance.  There is far more flexibility in choosing $\nu^\sharp$ however: it suffices to let $\nu^\sharp$ assign the majority of the variables $x_{i,1},\dots,x_{i,2k+1}$ the value of~$x_i$, and let remaining variables be assigned arbitrarily: this is because each $(k+1)$-subset of $\{1,\dots,2k+1\}$ will also include one of the elements in this majority.  Thus, given $\nu$ for $I$, we may speak of a ``choice for $\nu^\sharp$'' to mean any such assignment agreeing in majority with the underlying value given by $\nu$.   
Based on this, we may easily show that if $I$ is NAE-satisfiable (by some truth assignment $\nu$, say), then $I_1$ is $m$-robustly satisfiable for any $m\leq k$.  To see this, let $p$ be any partial assignment from some $m$-element set of variables $Y\subseteq X_k$ to $\{0,1\}$.  As $m<k+1$, the values of $p$ leave the majority of $x_{i,1},\dots,x_{i,2k+1}$ unassigned, and so we can choose $\nu^\sharp$ to agree with $p$ on all variables in $Y$ and still enable majority agreement with the underlying value given by $\nu$.  Thus $p$ extends to a NAE-satisfying truth assignment for $I_1$.

Now we must prove that when $I$ is a NO instance, then $I_1$ is a NO instance.  In the contrapositive, assume that~$\mu$ is a NAE-satisfying truth assignment for $I_1$.  
We claim that the truth assignment $\mu^\flat$ for~$I$ given by assigning each $x_i\in X$ the majority value in ${\mu}(x_{i,1}),\dots, {\mu}(x_{i,2k+1})$ is a NAE-satisfying assignment for $I$.  
Indeed, for each clause $(\bar{x}\vee\bar{y}\vee\bar{z})$ in $I$, consider $\{k+1\}$-element subsets $\{i_1,\dots,i_{k+1}\}$, $\{i_1',\dots,i_{k+1}'\}$, $\{i_1'',\dots,i_{k+1}''\}$ of $\{1,\dots,2k+1\}$ such that ${\mu}(\bar{x}_{i_1})=\dots={\mu}(\bar{x}_{i_{(k+1)}})={\mu}^\flat(\bar{x})$ and ${\mu}(\bar{y}_{i_1'})=\dots={\mu}(\bar{y}_{i_{(k+1)}'})={\mu}^\flat(\bar{y})$ and ${\mu}(\bar{z}_{i''_1})=\dots={\mu}(\bar{z}_{i''_{(k+1)}})={\mu}^\flat(\bar{z})$, which exist due to the definition of ${\mu}^\flat$.  
As the clause 
\[
(\bar{x}_{i_1}\vee\dots \vee \bar{x}_{i_{k+1}}\vee \bar{y}_{i_1'}\vee\dots \vee \bar{y}_{i_{k+1}'}\vee \bar{z}_{i_1''}\vee\dots \vee \bar{z}_{i_{k+1}''})
\]
exists in $I_1$ and is NAE-satisfied by $\mu$, it follows that $|\{\mu^\flat(\bar{x}),\mu^\flat(\bar{y}),\mu^\flat(\bar{z})\}|=2$ so that $(\bar{x}\vee\bar{y}\vee\bar{z})$ is NAE-satisfied by $\mu^\flat$, as required.
\end{proof}

{\begin{remark}\label{rem:distinct}
Theorem \ref{thm:NAE3k3} can be tightened by the assumption that each clause of $I_1$ contains at most one occurrence of each variable.
\end{remark}}
\begin{proof}
{The given reduction from NAE$3$SAT will not produce a repeated variable in a clause of $I_1$ if the clauses of $I_0$ are without repeats.  To prove the \texttt{NP}-completeness of positive NAE3SAT for this class of instances, first observe that NAE3SAT instances containing clauses of the form $(x\vee x\vee x)$, are trivially unsatisfiable in the NAE sense, so there is no loss of generality in assuming these do not appear.  Up to permutation of the order of appearance of variables, the only other clause type that contains a repeat is one of the form $(x\vee x\vee y)$ (for distinct $x$, $y$), which is satisfied in the NAE sense if and only if $x$ takes the opposite truth value to $y$.  We can replace such clauses by introducing new variables $a,b,c$ (for each such clause) and replacing $(x\vee x\vee y)$ by the  four clauses 
\[
(x\vee y\vee a), (x\vee y\vee b),(x\vee y\vee c),(a\vee b\vee c)
\]
which are equivalently satisfiable in the NAE sense.}
\end{proof}

{In \cite{AGK} it is stated that the standard reduction from SAT to 3SAT translates the $3$-robust 12SAT problem to the $3$-robust 3SAT problem; a sketch proof approach is suggested, while a complete argument can be found in the work of the author and Ham \cite[\S10]{hamjac}}.
We wish to establish an analogous result for the  NAE3SAT problem, this time in reduction from $\leq k$-robust NAE$(3k+3)$SAT for sufficiently large $k$; we find that $k=6$ suffices.

The standard reduction from $m$SAT (for $m>3$) to 3SAT replaces clauses $(\ell_1\vee\ell_2\vee \dots\vee\ell_{m})$ by the conjunction of $3$-clauses
\[
(\ell_1\vee \ell_2\vee z_2)\wedge (\neg z_2 \vee \ell_3\vee z_{3})\wedge \dots \wedge (\neg z_{m-3} \vee \ell_{m-2}\vee z_{m-2})\wedge (\neg z_{m-2}\vee \ell_{m-1}\vee\ell_m)
\]
by way of the introduction of $m-3$ new variables $z_2,\dots z_{m-2}$.  Starting from the instance $I_1$ constructed above {(an instance of positive NAE$(3k+3)$SAT)}, we let the new instance of NAE3SAT {created by this construction be denoted by $I_2$.}    Again this reduction can be computed in logspace. We refer to the {introduced variables as the \emph{$Z$-variables}, and the literals formed from these as \emph{$Z$-literals}}.  Variables in $X_k$ are referred to as \emph{$X$-variables}.   {(Not to be confused with the original set $X$ of variables for $I_0$, which we make no further mention of.)}
Note that the $z_i$ are chosen afresh for each clause (so that a more careful though cumbersome nomenclature would include the clause name).  We refer to $(\ell_1\vee\dots\vee \ell_{m})$ as the \emph{ancestor clause} for the displayed family of $3$-clauses, and that the displayed family of $3$-clauses are the {\emph{clause family} that is \emph{descendent} from $(\ell_1\vee\dots\vee \ell_{m})$}.

It is well known that {the construction just given} is a valid reduction from NAE{$m$}SAT to NAE3SAT {(for $m>3$)}{, so that $I_2$ is a NO instance when $I_1$ is a NO instance}.  {We are aiming to show that if $I_1$ if $\leq k$-robustly satisfiable then $I_2$ is ${\leq}3$-robustly satisfiable.}
In order to give the careful analysis required to translate robust satisfiability properties across this reduction we introduce a notation.
Let $p$ be a partial assignment from some domain $D\subseteq X_k\cup Z$.  \ 
Each $Z$-variable $z$ appears just twice in $I_2$, and only in a pattern of the form $(\underline{\phantom{a}}\vee\underline{\phantom{a}}\vee z)\wedge (\neg z\vee\underline{\phantom{a}}\vee\underline{\phantom{a}})$ within some clause family.  
If $z\in D$ and $p(z)=0$ then we write a $>$ above the bracket pair~``$)\wedge ($'' between the two occurrences of $z$, as follows: $(\underline{\phantom{a}}\vee\underline{\phantom{a}}\vee 0\stackrel{>}{)\wedge (}1\vee\underline{\phantom{a}}\vee\underline{\phantom{a}})$. 
If $p(z)=1$ then we similarly place a $<$ sign.  We will refer to a ``$>$ assignment'' to mean a $Z$-variable given the value $0$, and dually for $<$.
Note that the order signs appear to be the opposite of what might be natural, as it presents visually as suggesting the falsehoods $0>1$ and $1<0$.  The role of the $<,>$ signs however is to record what is needed to be witnessed elsewhere on account of the value that~$z$ has taken, rather than to record the relative ``size'' of the values taken by $z,\neg z$.  
Depending on the number of $Z$-literals from $D$ that appear in a clause family, the partial assignment~$p$ may contribute anywhere between $0$ and $m-1$ instances of $<,>$.  
As we will argue shortly, whenever a clause family contains a $>$ followed later by a $<$, we will need to be able to identify a variable appearing within this interval that takes the value~$0$.  
Likewise, if $<$ is followed by $>$ then we will need to identify a value $1$ somewhere in between.  
The left and right boundaries of the clause family will also act as conditional order symbols of both kinds, provided that the clause family includes at least one $Z$-literal assigned a value by $p$.  So,  if $>$ appears due to a variable $z\in D$ with $p(z)=0$, then there must be an $X$-variable taking the value $1$ somewhere to the left and an $X$-variable taking the value $0$ somewhere to the right; and dually for $<$ if $p(z)=1$. 
We will say that $z$ is \emph{balanced} by $p$ if there are $X$-variables in $D$ that provide the appropriate witnesses: between any pair of~$>$ and 
$<$ there is an $X$-variable taking the value $0$ under $p$, and dually for $<$ and $>$.  
Consecutive symbols in the sequence of order symbols for a single clause family under $p$ will be called \emph{adjacent}.  From the definition of balanced it should be clear that to balance $p$ it suffices to find appropriate values for $X$-variables that lie between adjacent symbols of different kinds: $>,<$ or $<,>$.  Two adjacent symbols will be called \emph{contiguous}, if they lie over brackets in the same clause, or equivalently, if there are no further $Z$-variables between them.

Before we give the key lemma that demonstrates the importance of balanced partial assignments, we give an example to demonstrate the definitions, here with $k=2$ and for space constraints, omitting $\wedge$ and writing $\vee$ as ``$,$'' only.  Consider the following clause family.
\begin{align*}
(x_0,x_1,z_1)(\neg z_1,x_2,z_2)(\neg z_2,x_3,z_3)(\neg z_3,x_4,z_4)(\neg z_4,x_5,z_5)(\neg z_5,x_6,z_6)(\neg z_6,x_7,x_8)
\end{align*}
Assume we have a partial assignment $p$ with domain $D=\{x_1,z_1,z_3,x_5,z_5,z_6\}$ and defined by:
\begin{center}
\begin{tabular}{c|cccccc}
$x$:&$x_1$&$z_1$&$z_3$&$x_5$&$z_5$&$z_6$\\
\hline
$p(x)$:&0&0&1&1&0&0
\end{tabular}
\end{center}
We obtain 
\begin{align*}
(x_0,0,\stackrel{>}{0)(1},x_2,z_2)(\neg z_2,x_3,\stackrel{<}{1)(0},x_4,z_4)(\neg z_4,1,\stackrel{>}{0)(1},x_6,\stackrel{>}{0)(1},x_7,x_8)
\end{align*}
This partial assignment is not balanced.  The pattern of order symbols is $><>>$ (only the $>>$ are contiguous), so allowing for the conditional order symbols at the boundaries, we would need to find $X$-variables witnessing $1>0<1>>0$.  Only one of these required values can be found amongst the present values: that given by $p(x_5)=1$.  The partial assignment can be extended to a balanced partial assignment in a number of ways:
\begin{center}\begin{tabular}{c|ccc}
$x$:&$x_0$&$x_2$ or $x_3$&$x_7$ or $x_8$\\
\hline
$p(x)$:&1&0&0
\end{tabular}
\end{center}
The value of $x_6$ is immaterial, and in general it is easily seen that a contiguous sequence of $>$ assignments to $Z$-literals already NAE-satisfies the clauses internal to the sequence regardless of the value of any $X$-variable within and likewise for~$<$.  Only at the changes between $>$ and $<$ (including the boundaries of clause families, which are conditionally both $<$ and $>$ for this purpose) does the value of the $X$-variable matter.
\begin{lem}\label{lem:balanced}
A partial assignment $p:D\to\{0,1\}$ extends to a NAE-satisfying total assignment of $I_2$ if and only if $p$ can be extended to balanced assignment $p'$ on a domain $D'\supseteq D$ with $D'\backslash D\subseteq X_k$ and such that $p'{\mid}_{D'\cap X_k}$ extends to a NAE satisfying assignment for $I_1$.
\end{lem}  
\begin{proof}
First let us assume that $p$ extends to a NAE satisfying assignment $p^+$ for~$I_2$.  
Then every $Z$-literal has been assigned, and it is obvious that whenever $>$ appears immediately before $<$ (and vice versa) the $X$-variable in the clause within must take the value $0$ (or $1$, respectively) for the clause to be NAE-satisfied.  
As usual, we include the boundaries of clause families as conditional $<$ and $>$.  
Thus for any adjacent occurrences of $<$, $>$ and $>$, $<$ within the same clause family under~$p$, there must be at least one $X$-variable assigned the desired value by $p^+$.  Fix such a witness for each pair of adjacent $<$, $>$ and $>$, $<$.   
If we restrict $p^+$ to the domain consisting of the union of $D$ and these $X$-variable witnesses, then we have a balanced extension~$p'$ of~$p$ to the domain $D'\supseteq D$, and such that $p'{\mid}_{D'\cap X_k}$ is a restriction of a NAE-satisfying assignment for $I_1$.

For the converse, assume that $p$ is balanced on the domain $D'$ and that $p{\mid}_{D'\cap X_k}$ extends to a NAE satisfying assignment $p^+$ for $I_1$.  
We may assume that $p^+$ not only extends $p{\mid}_{D'\cap X_k}$, but also extends $p$, by letting $p^+$ agree with $p$ on all $Z$-variables in $D'$.  Then $p^+$ is  a balanced partial assignment for $I_2$, missing only, possibly, some values on $Z$-variables. 
We now show how to extend $p^+$ to a NAE satisfying assignment for $I_2$.  
If $p^+$ is already total, then we are done, but otherwise there is at least one $Z$-variable not in $D'$.  
It suffices to show that we can extend $p^+$ to one further $Z$-variable and remain balanced; if so, then repeated application of this process eventually leads to a total assignment that is balanced and therefore NAE-satisfying.  
For instances where $>$ is adjacent but not contiguous to $>$, we may extend to all intermediate $Z$-variables by making further $>$ assignments.  
A dual statement holds for $<$ adjacent to $<$.  
Now we consider where an instance of $>$ is adjacent to but not contiguous to $<$ under $p^+$, with the dual case of $<$ adjacent to $>$ following by symmetry.  
We include the situation where both of $>$ and $<$ are from boundaries of clause families in this consideration, which is stronger than what is required for the notion of balanced (where we include boundaries as contributing conditional order symbols only to match instances that are created from $Z$-variables).  
Now, $p^+$ must assign $0$ to at least one $X$-variable $x$ between the current choice of $>$ and $<$: if both of $>$ and $<$ are from boundaries, then this is because $p^+$ NAE-satisfies $I_1$, but otherwise it is because $p^+$ is balanced.  
As $>$ and $<$ are adjacent but not contiguous, at most one of the symbols $>$, $<$ lies above a bracket in the clause containing $x$.  
Without loss of generality, assume that $>$ does not. 
We consider two cases, according to whether $>$ is from a boundary, or from a $Z$-literal.  
First consider when $>$ is due to a $Z$-variable assignment rather than a boundary, as in the following example (here with one boundary case at the other end, though this has no influence on the way that we extend $p^+$):
\begin{align*}
\ldots,\stackrel{>}{0)(1},\stackrel{\text{not }x}{\overbrace{\underline{\phantom{a}}}},z)(\neg z,\overbrace{\underline{\phantom{a}},z')(\neg z',\underline{\phantom{a}},\underline{\phantom{a}}}^{x\text{ here}}\stackrel{<}{)}.
\end{align*}
In this case we make a further $>$ assignment to the variable $z$, which maintains balance for $p^+$ given that $x$ appears somewhere to the right of $z$.  
The general case for this configuration differs from the example only in that there may be extra clauses between the one containing $z$ and the one containing $z'$, but the argument is identical.
Now consider when $>$ is due to a boundary, as in the following example (again, differing from the general case, only by having $z'$ appear in the second clause, but the argument is the same if there are further intermediate clauses):
\begin{align*}
\stackrel{>}{(}\stackrel{\text{not }x}{\overbrace{\underline{\phantom{a}}}},\stackrel{\text{not }x}{\overbrace{\underline{\phantom{a}}}},z)(\neg z,\overbrace{\underline{\phantom{a}},z')(\neg z',\underline{\phantom{a}}}^{x\text{ here}},1\stackrel{<}{)(}0,\ldots
\end{align*}
Let $y_0,y_1$ denote the two $X$-variables in the boundary clause associated with $>$ and $z$ the $Z$-literal (as shown in the example).  If $p^+(y_0)\neq p^+(y_1)$ or $p^+(y_0)=p^+(y_1)=1$ then we extend $p^+$ to $z$ by $p^+(z)=0$ as shown:
\begin{align*}
\stackrel{>}{(}1,1,0\stackrel{>}{)(}1,\underline{\phantom{a}},z')(\neg z',\underline{\phantom{a}},1\stackrel{<}{)(}0,\ldots
\end{align*}
which remains balanced because $p^+(x)=0$ still appears between the adjacent $>$, $<$.  If $p(y_0)=p^+(y_1)=0$ then we extend $p^+$ to $z$ by $p^+(z)=1$:
\begin{align*}
\stackrel{>}{(}0,0,1\stackrel{<}{)(}0,\underline{\phantom{a}},z')(\neg z',\underline{\phantom{a}},1\stackrel{<}{)(}0,\ldots
\end{align*}
which again is balanced.  This completes the proof, as repeated application eventually leads $p^+$ to be totally defined and NAE-satisfying.
\end{proof}
Each $Z$-variable in the domain of a partial assignment requires at most two $X$-variables to stabilise it, so if the size of $D$ is relatively small in comparison to $k$ then we have some hope that we can extend it to a balanced assignment on $D'$ containing at most $k$ distinct $X$-variables.  This is not immediately guaranteed, as the pattern of occurrences of $X$-variables around any $Z$-variables in $D$ have to be compatible with the values they are given by $p$.  When $|D|\geq 4$ it is easy to see that this may not in general be possible.  The following theorem shows that it is always possible when $|D|\leq 3$.

\begin{thm}\label{thm:3robNAE}
Let $k\geq 6$.  If $I_1$ is not NAE-satisfiable then $I_2$ is not NAE-satisfiable.  If $I_1$ is NAE-satisfiable then $I_2$ is ${\leq}3$-robustly NAE-satisfiable.
\end{thm}
\begin{proof}
The first statement is well known but also follows from Lemma \ref{lem:balanced}: if $\nu$ NAE-satisfies $I_2$, then it is balanced, from which it immediately follows that each clause family contains an $X$-variable given $0$ by $\nu$ and an $X$-variable given $1$ by $\nu$, so that $\nu$ also NAE-satisfies $I_1$.  

Now assume that $I_1$ is NAE-satisfiable, in which case we know additionally that every partial assignment on at most $k$ variables extends to a full NAE-satisfying assignment. 
Let $D\subseteq X_k\cup Z$ be a domain of ${\leq}3$ variables and $p:D\to \{0,1\}$ be any assignment that does not directly violate a clause of $I_2$.    We show that $p$ extends to a NAE-satisfying assignment for $I_2$.  First, it is trivial that $p$ can be extended to $3$ variables without violating a clause; so there is no loss generality in assuming $|D|=3$.  By Lemma \ref{lem:balanced} it now suffices to show that $p$ can be extended to a balanced partial assignment on some domain $D'\supseteq D$ containing at most $6$ distinct $X$-variables.    If $D$ contains no $Z$-variables, then this is trivial, as $p$ is already balanced.

In the remaining cases there will be $Z$-variables that will require balancing.  
We use the order symbol notation, and note that each pair of adjacent order symbols (including boundaries) bracket some $X$-variables, which we will refer to as an \emph{interval}.  
The \emph{width} of an interval will be the number of unassigned $X$-variables in the interval.  
We say that an interval is \emph{wide} (for a domain $D$) if it has width $\geq 6-|D\cap X_k|$.

Case 1: \emph{$D$ contains just one $Z$-variable, $z$, and two $X$-variables $x,y$}.\\
The variable $z$ appears exactly twice in $I_2$, in consecutive clauses within one clause family.  
The definition of balanced ensures that we need at most $2$ further $X$-variables to be added to $D$ in order to balance $p$: one for the interval left of $z$ and for the interval right $\neg z$.  
If one of these intervals (say, the left of $z$) has width $0$, then it is because $z$ appears in a boundary clause with $x,y$, and the assumption that $p$ does not violate this clause ensures it is already balanced; in this case the other interval has width at least $3k+1\geq 19$ (so, is wide) and we may choose any $X$-variable in it to balance $p$.  
If both intervals have width at least $1$, then we may select an unassigned $X$-variable from each interval and use these to balance $p$.  
Note that the assumption that $X$-variables appear at most once in each clause family is used here; see Remark \ref{rem:distinct}.

Case 2: \emph{$D$ contains two $Z$-variables, $z,z'$ and one $X$-variable $x$}.\\
  If $z,z'$ appear in the same clause family, then they create three intervals within the clause family; we may assume without loss of generality that $z$ appears left of $z'$.  
  In this case we need to select up to three further $X$-variables to balance $p$, one for each interval.  
  If all of these intervals contain an unassigned $X$-variable, then we can balance $p$ by extending $D$ to include these.  
  Otherwise, if one of the intervals has width $0$, it can only be the interval between $z$ and $z'$, and this only occurs in the case where $(\neg z\vee x\vee z')$ is a clause.  
  The assumption that $p$ does not violate $(\neg z\vee x\vee z')$ ensures that this interval is already balanced, so that we can balance~$p$ by selecting unassigned $X$-variables from the first and third interval.  
  
Next consider when $z$ and $z'$ appear in different clause families.  
In this case there are four intervals, though none of them can have width $0$, and at least two of the intervals are wide, having length at least $\lceil(3k+3)/2\rceil-1\geq 10$ (given that $k\geq 6$).  
Our strategy is to balance each interval starting from the shortest.  If there is at most one interval of width less than $2$ then this is easy: balance the shortest one (so that remaining shorter interval now has width at least $1$), then the next shortest.  
The remaining two intervals now still have width at least $10-2=8$, so that we may easily choose two further $X$-variables to balance $p$.  
  
The final consideration for Case 2 is when there are two intervals of length $1$, which means that both $z$ and $z'$ appear (possibly negated) in boundary clauses containing $x$.  
If one of these boundary clauses is already balanced by the value taken by $x$ then we may proceed as before by selecting unassigned $X$-variables to balance the remaining two or three intervals.  
If $p(x)$ balances neither of these two boundary clauses then it follows that the literal for $z$ appearing with $x$ takes the same value under $p$ as the literal for $z'$ appearing with $x$.  
Select the remaining $X$-variable $y$ and $y'$ from each of these boundary clauses and use them to balance the two boundary clauses by setting $p(y)=p(y')=\neg p(x)$.   
Note that this does not require  $y$ to be different from $y'$, as all $X$-variables appear in non-negated form in~$I_2$.  
The remaining two intervals have width at least $3k+1$ and can be balanced by the inclusion of two further $X$-variables (making a total of at most five $X$-variables in this case).

Case 3: \emph{$D$ contains only $Z$-variables $z,z',z''$}.\\
The subcase where all three appear in the same clause family is easy: there are four intervals and we may choose an $X$-variable from each to balance $p$.  If exactly two of $z,z',z''$ appear in the same clause family then we have $5$ intervals to balance, as depicted in the following schematic, with the various intervals labelled $A_1$--$A_5$:
\begin{multline*}
\stackrel{A_1}{\overbrace{\underline{\rule{0cm}{0.8em}\hspace{2cm}}}}z)\wedge (\neg z\stackrel{A_2}{\overbrace{\underline{\rule{0cm}{0.8em}\hspace{2cm}}}}z')\wedge (\neg z'\stackrel{A_3}{\overbrace{\underline{\rule{0cm}{0.8em}\hspace{2cm}}}}\\
\stackrel{A_4}{\overbrace{\underline{\rule{0cm}{0.8em}\hspace{2cm}}}}z'')\wedge (\neg z''\stackrel{A_5}{\overbrace{\underline{\rule{0cm}{0.8em}\hspace{2cm}}}}
\end{multline*}
None of the intervals can have width $0$,  only $A_2$ can have width $1$, and one of $A_4$ and $A_5$ has width at least $\lceil(3k+3)/2\rceil-1\geq 10$ and there will be no loss of generality in assuming that $A_4$ is the longer interval.  First choose an $X$-variable,~$x_2$,  to balance~$A_2$.  If~$x_2$ appears in any other interval, it can only be in one of $A_4$ or~$A_5$, which are reduced in width by at most $1$ by the value given to $x_2$.  Next choose an unassigned $X$-variable, $x_5$, to balance $A_5$, which is possible as $A_5$ has width at least $1$ at this stage.  Note that $A_4$ still has width $\geq 9$ at this point.  Now both of $A_1$ and $A_3$ have width at least one (with $1$ only possible if $x_5$ appears).  Balance $A_1,A_3$ by selecting unassigned variables $x_1,x_3$.  The interval $A_4$ still has width at least $7$, so can be balanced by selection of one further unassigned $X$-variable.  Then $p$ is balanced.  

Finally, assume that all three of $z,z',z''$ appear in different clause families:
\begin{multline*}
\stackrel{A_1}{\overbrace{\underline{\rule{0cm}{0.8em}\hspace{2cm}}}}z)\wedge (\neg z\stackrel{A_2}{\overbrace{\underline{\rule{0cm}{0.8em}\hspace{2cm}}}}\\
\stackrel{A_3}{\overbrace{\underline{\rule{0cm}{0.8em}\hspace{2cm}}}}z')\wedge (\neg z'\stackrel{A_4}{\overbrace{\underline{\rule{0cm}{0.8em}\hspace{2cm}}}}\\
\stackrel{A_5}{\overbrace{\underline{\rule{0cm}{0.8em}\hspace{2cm}}}}z'')\wedge (\neg z''\stackrel{A_6}{\overbrace{\underline{\rule{0cm}{0.8em}\hspace{2cm}}}}
\end{multline*}
As before we must choose an $X$-variable from each of the six intervals $A_1$--$A_6$ to balance $p$.  Up to symmetry, there is no loss of generality in assuming $A_1$ is narrower than $A_2$, that $A_3$ is narrower than $A_4$ and $A_5$ is narrower than $A_6$.  Note  that $A_2$, $A_4$ and $A_6$ are wide intervals, here meaning they initially have at least six unassigned $X$-variables. 

For $i=1,3,5$, let $x_i,y_i$ be the leftmost $X$-variables in $A_i$, noting that by Remark~\ref{rem:distinct} we have that $x_i\neq y_i$.  If $|\{x_1,y_1,x_3,y_3,x_5,y_5\}|\geq 3$ we may find three distinct representatives of $\{x_1,y_1,x_3,y_3,x_5,y_5\}$, one each from $\{x_1,y_1\}$, $\{x_3,y_3\}$ and $\{x_5,y_5\}$, and balance $A_1$, $A_3$ and $A_5$ using these.  If $|\{x_1,y_1,x_3,y_3,x_5,y_5\}|\leq 2$ then $\{x_1,y_1\}=\{x_3,y_3\}=\{x_5,y_5\}$ and it suffices to assign $0$ to $x_1$ and $1$ to $y_1$ to balance all three of $A_1,A_3,A_5$.  Now we may balance the wider intervals $A_2,A_4,A_6$ easily, as they each have at least three remaining unassigned $X$-variables.  This completes Case 3, the final case.
\end{proof}

The following lemma is an immediate corollary of Theorem \ref{thm:3robNAE} and the fact that three variables can appear together (negated or non-negated) in at most one clause of $I_2$.  
\begin{lem}\label{lem:atmost2}
Let $v_1,v_2,v_3$ be three variables from $X_k\cup Z$, not necessarily distinct and assume that the instance $I_2$ is NAE-satisfiable.  
\begin{itemize}
\item If $v_1,v_2,v_3$ are distinct variables that appear together in a single clause, then there are precisely two assignments $p:\{v_1,v_2,v_3\}\to\{0,1\}$ that do not extend to NAE-satisfying assignments for $I_2$, and these assignments are negations of each other.  
\item Otherwise, if $|\{v_1,v_2,v_3\}|\leq 2$ or if $v_1,v_2,v_3$ do not appear within a single clause, then every assignment $p:\{v_1,v_2,v_3\}\to\{0,1\}$ extends to a NAE-satisfying assignment for $I_2$.
\end{itemize}
\end{lem}

\section{Edge and vertex robust G3C}\label{sec:G3C}
The main result of this section is {the following} Theorem \ref{thm:G3C}, which will be an immediate corollary of Theorem \ref{thm:G3CI} below and Theorem \ref{thm:3robNAE}.  A \emph{triangle} in a graph will mean a $3$-clique.
\begin{thm}\label{thm:G3C}
{The} ${\leq} 2$-robust graph $3$-colourability {problem} is \texttt{NP}-complete.  Moreover, this remains true for graphs in which every edge lies within exactly one triangle.
\end{thm}
The first statement in the theorem is given in \cite[Theorem 3.3]{AGK}.  {The second statement will be required for later results in the article.}  In Proposition \ref{pro:edgerobust} we also observe a slight extension to the theorem, that shows in the case that our construction is $3$-colourable, then there are also no non-local restrictions to the number of colours required to cover two distinct edges. 

The construction is closely based on a standard reduction of NAE3SAT to G3C except we have more triangles, and the robust colourability is proved by way of Theorem \ref{thm:3robNAE}.  

We continue to build {on the reduction commenced in} Section \ref{sec:sat}, taking {the} instance $I_2$ of NAE3SAT, satisfying the conditions of Theorem \ref{thm:3robNAE} {and Lemma \ref{lem:atmost2}}.
We must construct a graph $I_3=\mathbb{G}_{I}$ (we write $\mathbb{G}_I$ rather than $\mathbb{G}_{I_2}$ to avoid double subscript, and because the construction ultimately depends on the original instance $I_0=I$ of positive NAE3SAT).
The vertices are as follows.
\begin{itemize}
\item[A1] There is a special vertex $a$.
\item[B1] For each variable $x\in X_k\cup Z$ there are vertices corresponding to both $x$ and $\neg x$ (the \emph{literal vertices}: we consider these to be both literals in $I_2$ and vertices in $\mathbb{G}_I$).
\item[C1] For each clause $e$ {of $I_2$} we have three \emph{connector} vertices $c_{e,1}, c_{e,2},c_{e,3}$.
\item[D1] For each clause $e$ {of $I_2$} we have three vertices $e_1,e_2,e_3$ (the \emph{clause position vertices}).
\end{itemize}
The edges are described as follows.
\begin{itemize}
\item[E1] The special vertex is adjacent to all literal vertices.
\item[F1] Each literal vertex $x$  is adjacent to its negation $\neg x$.
\item[G1] For each clause $e$, the clause vertices $e_1,e_2,e_3$ form a {triangle}.
\item[H1] For each clause $e=(\ell_1\vee\ell_2\vee\ell_3)$ and each $i\in \{1,2,3\}$, the following three vertices form a triangle: the literal vertex $\ell_i$, the connector vertex $c_{e,i}$ and $i^{\rm th}$ clause position vertex $e_i$ for $e$.
\end{itemize}
Figure~\ref{fig:G} gives a schematic of this construction, with connector vertices shown with smaller vertex symbols.  In a typical instance, some type B1 (literal) vertices will be involved in many edges, but those corresponding to negations of variables from~$X_k$ will  connect only to vertex $a$ and to their non-negated form, as variables from~$X_k$ appear only positively in $I_2$.
\begin{figure}
\begin{tikzpicture}
\node [element, minimum size=6pt,very thick, label = $a$] (a) at (0,4) {};
\foreach \x in {0,...,5}{
\node [element, minimum size=6pt,very thick] (\x) at (\x/2-2,3) {};
\draw [thick] (a) to (\x);
}
\node at (1.15,3) {\ldots};
\foreach \x in {6,7}{
\node [element, minimum size=6pt,very thick] (\x) at (\x/2-1.3,3) {};
\draw [thick] (a) to (\x);
}
\foreach \x/\y in {0/1,2/3,4/5,6/7}{
\draw [thick] (\x) to (\y);
}

\node at (1,1.5) {\ldots};

\foreach \x in {0,1,3}{
\node [element, minimum size=6pt,very thick] (tu\x) at (1.4*\x-2.1,2) {};
\node [element, minimum size=6pt,very thick] (tl\x) at (1.4*\x-2.5,1.2) {};
\node [element, minimum size=6pt,very thick] (tr\x) at (1.4*\x-1.7,1.2) {};
\draw [thick] (tu\x) to (tl\x) to (tr\x) to (tu\x);
}

\draw [thick] (tl0) [out = 90,in = 240] to (0);
\node [element, minimum size=3pt, thick] (0a) at (-2.6,2.2) {};
\draw (tl0) [out = 95,in = 270] to (0a) [out = 60,in = 230] to (0);
\draw [thick] (tu0) [out = 50,in = 230] to (2);
\node [element, minimum size=3pt, thick] (2a) at (-1.75,2.55) {};
\draw (tu0) [out = 60,in = 230] to (2a) [out = 30,in = 220] to (2);

\draw [thick] (tr0) [out = 90,in = 220] to (4);
\node [element, minimum size=3pt, thick] (4a) at (-1.38,2.25) {};
\draw (tr0) [out = 95,in = 235] to (4a) [out = 30, in = 215] to (4);

\draw [thick] (tl1) [out = 110,in = 220] to (5);
\node [element, minimum size=3pt, thick] (5a) at (-0.5,2.5) {};
\draw  (tl1) [out = 120,in = 210] to (5a) [out = 25,in = 215] to (5);
\draw [thick] (tu1) [out = 50,in = 220] to (1.1,2.5);
\draw (tu1) [out = 40,in = 220] to (1.2,2.5);
\draw [thick] (tr1) [out = 90,in = 220] to (0.8,2.5);
\draw  (tr1) [out = 85,in = 220] to (0.9,2.5);

\draw [thick] (tl3) [out = 110,in = 280] to (2);
\node [element, minimum size=3pt, thick] (2aa) at (-.2,2) {};
\draw  (tl3) [out = 112,in = 350] to (2aa) [out = 160,in = 270] to (2);

\draw [thick] (tu3) [out = 130,in = 300] to (1.4,2.5);
\draw [thick] (tr3) [out = 60,in = 300] to (7);
\node [element, minimum size=3pt, thick] (7a) at (2.8,2.2) {};
\draw  (tr3) [out = 50,in = 270] to (7a) [out = 105,in = 315] to (7);
\draw (tu3) [out = 125,in = 300] to (1.5,2.5);

\node at (-3.6,4) {A1:};
\node at (-3.6,3.05) {B1:};
\node at (-3.6,2.3) {C1:};
\node at (-3.6,1.5) {D1:};

\end{tikzpicture}
\caption{A schematic of $\mathbb{G}_{I}$}\label{fig:G}
\end{figure}
\begin{remark}\label{rem:4cycles}
The property that no clause of $I_2$ contains a variable {twice (even negatively and positively)} ensures that there are no $4$-cycles in $\mathbb{G}_{I}$: more precisely there are no four distinct vertices $u_0,u_1,u_2,u_3$ with each of $\{u_i,u_{i\oplus_4 1}\}$ an edge.  
It is also clear that each edge of $\mathbb{G}_I$ lies in exactly one triangle.
\end{remark}
\begin{thm}\label{thm:G3CI}
If $I_2$ is not NAE-satisfiable then $\mathbb{G}_I$ is not $3$-colourable.  If $I_2$ is NAE-satisfiable, then $\mathbb{G}_I$ is ${\leq} 2$-robustly $3$-colourable.
\end{thm}
\begin{proof}
We begin by setting some notation that is useful through the proof.  We abbreviate the four vertex types (A1), (B1), (C1), (D1) as $A$, $B$, $C$, $D$ for brevity in the remainder of this proof.
For a vertex  $w$ of type $C$ and $D$, we let $w_B$ denote the unique literal associated with $w$, $w_C$ the unique connector vertex equal or adjacent to $w$, and $w_D$ the unique position vertex equal or adjacent to $w$ (if $w$ has type~$C$). 
We also let $e_w$ denote the clause corresponding to the triangle of position vertices that includes $w_D$.
If $w$ has type $B$, then we also use $w_B$ to denote $w$, though leave $w_C$, $w_D$ and $e_w$ undefined in general.

Aside from the connector vertices, which play only a triangulation role, the construction of $\mathbb{G}_I$ from $I_2$ is a standard reduction, but we will need to explore one half of the reduction more carefully than usual.
For the easy direction: when there is a $3$-colouring of $\mathbb{G}_I$, we may permute the colour names so that the special vertex~$a$ has been coloured $2$; then by (E1), the colours of the literal vertices lie amongst $\{0,1\}$, which can then be shown to be a NAE-satisfying truth assignment for $I_2$.  
This fact is completely standard and routinely verified from the edge definitions (F1), (G1).  
The other direction is also standard but we need the details for the robust satisfiability argument.  

Assume that $I_2$ has been NAE-satisfied by some truth assignment $\nu$.  
We  define the following \emph{colouring strategy} for $3$-colouring $\mathbb{G}_I$. 
\begin{itemize}
\item Colour the vertex $a$ by $2$.  
\item Colour each literal vertex by its truth value.  
\item For each clause $e=(\ell_1\vee\ell_2\vee\ell_3)$ there are two literals taking the same value under $\nu$ (the \emph{majority value}).  
\begin{itemize}
\item Make an arbitrary choice of one of the literals taking the majority value and colour its position vertex by $2$.
\item For the other two vertices, colour the position vertex by the opposite of the truth value the corresponding literal takes.  
\end{itemize}
\item Each connector vertex is adjacent to just two other vertices: a literal vertex and a clause position vertex, both already coloured.  Colour the connector vertex by the remaining available colour.  
\end{itemize}
The colouring strategy gives a valid $3$-colouring of $\mathbb{G}_I$.  
To show that $\mathbb{G}_I$ is ${\leq}2$-robustly $3$-colourable, we must take this line of argument further, utilising the fact that in the situation that $I_2$ is NAE-satisfiable, it is also ${\leq}3$-robustly satisfiable.
 
Let $u,v$ be a pair of distinct  vertices.  
Our goal is to show that any valid colouring of $u$ and $v$ extends to a $3$-colouring for $\mathbb{G}_I$.  
We will say that the vertices $u$ and $v$ are \emph{like-coloured} if they are given the same colour, and are \emph{unlike-coloured} if they are given different colours.
It  suffices to find a $3$-colouring of $\mathbb{G}_I$ that like-colours $u$ and $v$ if they are not adjacent, and a $3$-colouring that unlike-colours $u$ and $v$, as all other colourings of $u,v$ can be obtained by permuting the colour names from these. 
If $u$ and $v$ are adjacent, then all colourings unlike-colour~$u$ and~$v$ and there is nothing further to check.  
Now assume that~$u$ is not adjacent to~$v$.  

Without loss of generality, we may assume that the vertex type ($A,B,C,D$) of~$u$  is alphabetically earlier or equal to the type of $v$.  
If $u$ has type $A$, then $u=a$, and as~$v$ is not adjacent to $u$, it follows that $v$ has type $C$ or~$D$.   
Using only the $2$-robust satisfiability of $I_2$ we may easily find a NAE-satisfying truth assignment placing the literal $v_B$ in the majority value within the clause $e_v$, allowing for the flexible choice of colour~$2$ (like-colouring with $u$) or a colour from $\{0,1\}$ for $v$ (unlike-colouring with $u$).
We may now assume that the type of $u$ is at least $B$.  

If both $u$ and $v$ have type $B$, then~$u$ and~$v$ are literals, and under the colouring strategy their colour is their truth value.  
In this case we again need only the $2$-robust NAE-satisfiability of $I_2$ to find a like- and an unlike-colouring of $u$ and $v$.
Thus we may assume that  $v$ has type either $C$ or~$D$, meaning that $v_B$, $v_C$, $v_D$ and~$e_v$ are all defined.

The clause $e_v$ corresponding to $v$ contains two further literals beyond $v_B$, at least one of which being neither  $u_B$, nor $\neg u_B$; we denote this literal by $w=w_B$, and let $w_C$ and $w_D$ denote the connector and position vertices corresponding to the occurrence of $w$ in the clause $e_v$.    
It is still possible that $u_B$ or $\neg u_B$ is equal to $v_B$ or is the third literal in $e_V$, though in the case that $u_B$ is $v_B$, then the assumption that $u$ is not adjacent to $v$ ensures that $u$ is of type~$C$ or~$D$. 
We now consider some possible truth values for a triple of literals $(u_B,v_B,\ell)$ including $u_B$, $v_B$ and some generic third literal $\ell$ distinct from $u_B$, $v_B$ and their negations.  
In most instances we will choose $\ell$ as $w$.  
Because the dual of a NAE-satisfying truth assignment is also NAE-satisfying, it will suffice to consider only assignments in which~$u_B$ takes the value~$0$, and the four combinations are gathered in columns $u_B$, $v_B$ and~$\ell$ of Table~\ref{tab:table}.
\begin{table}
\begin{tabular}{c|ccc|cccc}
&$u_B$&$v_B$&$\ell$&$u_C$&$u_D$&$v_C$&$v_D$\\
\hline
$p_0$&0&0&0&2&1&1/2&2/1\\
$p_1$&0&0&1&2&1&2&1\\
$p_2$&0&1&0&2&1&2&0\\
$p_3$&0&1&1&2&1&0/2&2/0\\
\end{tabular}
\caption{\rule{0cm}{1.4em}The four possible partial assignments $p_0,p_1,p_2,p_3$ on $(u_B,v_B,\ell)$ with $p_i(u_B)=0$ and some  possible colourings of $u_C$, $u_D$, $v_C$, $v_D$ under the colouring strategy.}\label{tab:table}
\end{table}
If $u_B=v_B$ then only $p_0$ and $p_1$ are well-defined, but by Lemma~\ref{lem:atmost2}, both extend to full NAE-satisfying assignments for~$I_2$.  
Similarly, if $\neg u_B=v_B$ then only $p_2$ and $p_3$ are well-defined, but again both extend to full NAE-satisfying assignments for~$I_2$.
In all other cases however, all four rows give meaningful partial assignments, though Lemma~\ref{lem:atmost2} only guarantees that at least three are locally compatible (and therefore extend to a NAE-satisfying truth assignment).
The remaining columns give possible colourings for $u_C,u_D,v_C,v_D$, though the validity of these will depend on the choice of $\ell$ and whether the assignment $p_i$ is locally compatible.  
Note that if $u=u_B$, then we did not define $u_C$ and~$u_D$, and so columns $u_C$ and~$u_D$ should be ignored.  
If $u\neq u_B$ (that is, $u$ has type $C$ or $D$), then 
possibly $u_D$ also lies in the component corresponding to $e_v$, 
but in that case we must have $u=u_C$, because $u=u_D$ would place $u$ adjacent to $v=v_D$ in the component for $e_v$.  
In this situation, where $u_D$ lies in the clause component for $e_v$, it follows that $e_v$ is the clause $u_B\vee v_B\vee w$ (up to commutativity of $\vee$) and the assignment of $(0,1,1)$ to $(u_B,v_B,w)$ is a locally compatible assignment giving $v_B$ and $w$ the majority value in  $e_v$.  
The colouring strategy then colours $u_C$ by $2$, and $v_C$ can be coloured either $2$ or $0$, and $v_D$ either $0$ or $2$ respectively.
This achieves the desired like- and  unlike-colouring of $u$ and $v$.  
From now we may assume that if $u$ has type $C$ or $D$, then $u_C$ and $u_D$ are associated with a different component to the one associated with $e_v$.  It then follows that any choices available for $u_C$ are independent of those for $v_C$.
In particular, the colouring strategy always allows for $u_C$ to be coloured~$2$, and $u_D$ to be coloured by $\neg u_B$.  This is what is shown in the corresponding columns of Table~\ref{tab:table}.

If in Table~\ref{tab:table} we choose $\ell$ to be $w$, then the colours shown for $v_C$ and $v_D$ are allowed by the colouring strategy.
In the case of $p_0$ and $p_3$, literals $v_B$ and $w=\ell$ have the same truth value, so the vertices $v_C$ and $v_D$ (and $w_C$ and $w_D$) get a flexible choice between the colour $\neg v_B$ and $2$ (this is shown as $1/2$ in the $p_0$ row and $0/2$ in the $p_3$ row).
In the case of $p_1$ and $p_2$, we cannot know which of $v_B$ and $w$ is taking the majority value for $e_v$, but it is always possible for $v_D$ to take the colour $\neg v_B$ and for $v_C$ to be coloured $2$, which is what is shown in the table.

\begin{table}
\begin{tabular}{c|ccccc}
$(u,v)$:&$(B,C)$&$(B,D)$&$(C,C)$&$(C,D)$&$(D,D)$\\
\hline
$=$:&$p_3$&$p_2,p_3$&$p_0,p_1$&$p_0,p_3$&$p_0,p_1$\\
$\neq$:&$p_0,p_1$&$p_0,p_1$&$p_0,p_3$&$p_0,p_1$&$p_0,p_2$
\end{tabular}
\caption{\rule{0em}{1.4em}Some choices of rows witnessing like-colouring and unlike-colouring of $u$, $v$, where $=$ denotes cases that can be used to like-colour $u$ and $v$, and $\neq$ denotes cases that can be used to unlike-colour $u$ and $v$.}\label{tab:2}
\end{table}

There are the following five possible pairs of vertex types for $(u,v)$:
\[
(B,C),(B,D),(C,C),(C,D),(D,D),
\]
and we now explore how many have been flexibly coloured by the colourings in Table~\ref{tab:table}, where $\ell=w$ unless otherwise stated.  
Table~\ref{tab:2} gives a sample of up to two choices from $p_0,\dots,p_3$ that like- and unlike-colour $u$ and $v$, provided they are well-defined.

Case 1. When $u_B=v_B$, only $p_0$ and $p_1$ are well-defined, but both extend to full NAE-satisfying assignments for $I_2$, so the given colourings also extend to $3$-colourings of $\mathbb{G}_I$.  
As $u$ is either $u_C$ or $u_D$ in this situation, the two choices offered by $p_0$ are already sufficient to colour $u$ and $v$ in the desired ways.

Case 2. When $\neg u_B=v_B$, then only $p_2$ and $p_3$ are well-defined, but again both extend to a full NAE-satisfying assignments for $I_2$, so the given colourings extend to all of $\mathbb{G}_I$.  
This time we achieve like and unlike-colouring for $u_C$ with respect to the two possibilities for $v$ using $p_3$, but we achieve only an unlike-colouring for the case where $u=u_D$ and $v\in \{v_C,v_D\}$ (which implies $v=v_D$ due to the alphabetical ordering assumptions).  
In order to like-colour $u_D$ and $v$ for this case, it is necessary to colour both by $2$.  
Let $e_u$ be the clause corresponding to the component for $u$, and let $w'$ be a literal in $e_u$ whose underlying variable is distinct from that of $v_B$ and $w$.  
Then the assignment $(v_B,w,w')\mapsto (1,1,0)$ places $u_B=\neg v_B$ in the majority value with $w'$ in $e_u$ and $v_B$ in the majority value with $w$ in $e_v$.  
Due to the independent choice of colour in $\{1,2\}$ for $u_D$ and $\{0,2\}$ for $v_D$, a  like-colouring of $u$ and $v$ is achieved provided this truth assignment is locally compatible for $I_2$.
The clause $e_u$ contains the negation of a literal in $e_v$ (namely $u_b=\neg v_B$), so $e_u$ and $e_v$ are neighbouring clauses within a clause family.  It follows that $v_B$ (and $u_B$) is a $Z$-literal, and then that there is no clause of $I_2$ that contains all three of $w,w',v_B$.  So every assignment to these three literals is locally compatible, as required.

Case 3.  $u_B$ is neither $v_B$ nor $\neg v_B$.  
In this case, at least three of the four partial assignments in Table~\ref{tab:table} extend to a full NAE-satisfying truth assignment for $I_2$ by Lemma \ref{lem:atmost2}.
As seen in Table~\ref{tab:2}, only the case of like-colouring $u_B$ and $v_C$ fails to achieve two distinct choices (at least one of which extends to a valid $3$-colouring under the colouring strategy). 
If the assignment $p_3$ extends, then the proposed like-colouring of $u_B$ and $v_C$ in the $p_3$ row extends to a full $3$-colouring of $\mathbb{G}_I$ and we are done.
If the assignment $p_3$ fails to extend, then the ${\leq}3$-robust satisfiability of $I_2$ implies that it must explicitly violate a clause of $I_2$.  
So, up to commutativity of $\vee$, either the clause $(\neg u_b\vee v_B\vee w)$ or $(u_b\vee \neg v_B\vee\neg w)$  appears in $I_2$.  Because these clauses agree with $e_v$ on the underlying variables of both $v_B$ and $w$, it follows that $v_B$ and $w$ are $X$-variables and appear only positively in $I_2$; this then shows that  $(\neg u_b\vee v_B\vee w)$ is the only possible clause to fail $p_3$.  
Now, $(\neg u_b\vee v_B\vee w)$  could be $e_v$ itself or it could be some other clause $e$.

Subcase 3.1.  If $(\neg u_b\vee v_B\vee w)$ is $e_v$, consider the assignment $p_2$, still with the choice of $\ell$ as $w$.  
Under this assignment, $v_B$ has the majority value in $e_v$ (with $\neg u_B$), so it is possible to colour $v_C$ by $0$ and $v_D$ by $2$, thereby achieving the like-colouring of $u$ and $v$.

Subcase 3.2.  If $(\neg u_B\vee v_B\vee w)$ is a distinct clause to $e_v$, then  $e_v$ is  $(w'\vee v_B\vee w)$ for some literal $w'$ with an underlying variable distinct from that of $u_B$ and $v_B$.  
In this situation, $u_b$ and $w'$ must be distinct $Z$-literals, and occurring in distinct clause families.
So there is no clause of $I_2$ containing all three variables for $u_B,v_B,w'$.  
Therefore, if $\ell$ denotes $w'$ (instead of $w$), the assignment $p_3$ is  locally compatible, and the colouring given in Table~\ref{tab:table} achieves the desired like-colouring of  $u_B$ and $v_C$.
  
As the last remaining case, the proof is complete.
\end{proof}

The following corollary of Theorem \ref{thm:G3CI} has been extended to almost complete generality by the All or Nothing Theorem (ANT) of \cite{hamjac}.  The present proof takes a more direct path and was a precursor step in understanding that paved the way to results such as the ANT.
\begin{cor}\label{cor:uHK3}
It is \texttt{NP}-complete to decide membership of graphs in the universal Horn class of $\mathbb{K}_3$.  More particularly, when the instance $\mathbb{G}_I$ is ${\leq} 2$-robustly $3$-colourable then $\mathbb{G}_I$ is in the universal Horn class of $\mathbb{K}_3$, while when $\mathbb{G}_I$ is not $3$-colourable then $\mathbb{G}_I$ admits no homomorphisms into  $\mathbb{K}_3$ and hence is not in the universal Horn class of $\mathbb{K}_3$.
\end{cor}
\begin{proof}
This follows immediately from Theorem \ref{thm:G3CI} and Example \ref{eg:K3}(iii).
\end{proof}
The All or Nothing Theorem of \cite{hamjac} implies that for every $k\in\mathbb{N}$ there exists an $\ell\in\mathbb{N}$ such that it is \texttt{NP}-hard to distinguish the non-$3$-colourable graphs from those satisfying the property that every partial $3$-colouring on $k$ vertices extends to a full $3$-colouring, provided that it can be extended to any $\ell$ further vertices.  
Theorem~\ref{thm:G3C} shows \texttt{NP}-hardness when $k=2,\ell=0$.  
For $k\geq 3$ it is easy to see that \texttt{NP}-hardness in the $3$-colouring case will require $\ell\geq 1$, as the neighbours of any vertex of degree at least $3$ can have at most $2$ colours in any $3$-colouring.  
Is $\ell=1$ enough for \texttt{NP}-hardness when $k=3$?  This seems to be unknown.  
In general, it would be interesting to understand the minimum value of $\ell$, as a function of $k$, for \texttt{NP}-hardness of $3$-colouring.  The same questions can also be asked for $n$-colouring (with $n\geq 3$).  
We provide two propositions giving tentative initial results in this style.  
Proposition \ref{pro:edgerobust} gives some partial results on the \texttt{NP}-hardness of extendability of partial $3$-colourings in the case of $k=4$ and $\ell=2$, but only in the case of colourings of two pairs of adjacent vertices.  
Proposition \ref{pro:highercolouring} concerns $\lfloor (n-1)/2\rfloor$-robust $n$-colourability when $n\geq 5$ and uses recent advances in the so-called promise-CSP.

Let $\{i_1,i_2\}$ and $\{i_1',i_2'\}$ be two edges of $\mathbb{G}_I$, not identical.  There are always at least two colours appearing amongst $\{i_1,i_2,i_1',i_2'\}$, but there are some simple local configurations that, when present, can force there to be always exactly $2$ colours, or always exactly $3$ colours: see Figure \ref{fig:config}.
\begin{figure}
\begin{tikzpicture}
\node at (0,0) {
\begin{tikzpicture}
\node (1) [draw,thick,circle,inner sep=1mm, fill = white] at (0,0) {$i_1$};
\node (2) [draw,thick,circle,inner sep=1mm, fill = white] at (0,1) {$i_2$};
\node (1') [draw,thick,circle,inner sep=1mm, fill = white] at (2,0) {$i_1'$};
\node (2') [draw,thick,circle,inner sep=1mm, fill = white] at (2,1) {$i_2'$};
\node (3) [draw,thick,circle,inner sep=1mm, fill = white] at (1,0.5) {$\exists$};

\draw (1) -- (2) -- (3) -- (1') -- (2') -- (3) -- (1);
\end{tikzpicture}};
\node at (4.5,0) {\begin{tikzpicture}
\node (1) [draw,thick,circle,inner sep=1mm, fill = white] at (0,0) {$i_1$};
\node (2) [draw,thick,circle,inner sep=1mm, fill = white] at (0,1) {$i_2$};
\node (1') [draw,thick,circle,inner sep=1mm, fill = white] at (3,0) {$i_1'$};
\node (2') [draw,thick,circle,inner sep=1mm, fill = white] at (3,1) {$i_2'$};
\node (3) [draw,thick,circle,inner sep=1mm, fill = white] at (1,0.5) {$\exists$};
\node (4) [draw,thick,circle,inner sep=1mm, fill = white] at (2,0.5) {$\exists$};

\draw (1) -- (2) -- (3) -- (4) -- (1') -- (2') -- (4) -- (3) -- (1);
\end{tikzpicture}};

\node at (9,0) {\begin{tikzpicture}
\node (1) [draw,thick,circle,inner sep=1mm, fill = white] at (0,0) {$i_1$};
\node (2) [draw,thick,circle,inner sep=1mm, fill = white] at (0,1.5) {$i_2$};
\node (1') [draw,thick,circle,inner sep=1mm, fill = white] at (1.5,0) {$i_1'$};
\node (2') [draw,thick,circle,inner sep=1mm, fill = white] at (1.5,1.5) {$i_2'$};

\draw (1') -- (2) -- (1) -- (1') -- (2');
\end{tikzpicture}};

\end{tikzpicture}
\caption{Three local configurations forcing the number of colours on $\{i_1,i_2,i_1',i_2'\}$ in any valid $3$-colouring (other edges and vertices may also be present but the $\exists$ symbol means that such a vertex does appear).  In the first, exactly two colours can appear in $\{i_1,i_2,i_1',i_2'\}$.  In the second and third, exactly $3$ colours appear in $\{i_1,i_2,i_1',i_2'\}$t.  In the third configuration, we allow $i_2'$ to coincide with one of $i_1$ or $i_2$.}\label{fig:config}
\end{figure}

\begin{pro}\label{pro:edgerobust}
Let $\{i_1,i_2\}$ and $\{i_1',i_2'\}$ be {two non-identical edges of $\mathbb{G}_I$}.  Except in the three configurations given in Figure \ref{fig:config} \up(and permutations of these labellings\up), if there is a $3$-colouring {of $\mathbb{G}_I$}, then there is a $3$-colouring giving $\{i_1,i_2,i_1',i_2'\}$ exactly~$2$ colours and a $3$-colouring giving $\{i_1,i_2,i_1',i_2'\}$ exactly $3$ colours.
\end{pro}
\begin{proof}
We show that in cases other than those of Figure \ref{fig:config}, we may achieve both the situation of exactly two colours amongst $\{i_1,i_2,i_1',i_2'\}$ and three colours amongst $\{i_1,i_2,i_1',i_2'\}$.

{Case 1. $|\{i_1,i_2,i_1',i_2'\}|\leq 3$.  As $\{i_1,i_2\}\neq \{i_1',i_2'\}$ and the third configuration from Figure \ref{fig:config} is being avoided, we may assume without loss of generality that $i_1=i_1'$, but $\{i_2,i_2'\}$ is not an edge}.  
In this case we can use $2$-robust $3$-colourability to like-colour {$i_2$ and $i_2'$} to achieve exactly two colours in the set, and unlike-colour {$i_2$ and $i_2'$} to achieve all three colours.

Case 2. $|\{i_1,i_2,i_1',i_2'\}|=4$.  
In this case, avoiding the first and third configurations (up to permutations of $\{i_1,i_2\}$ and $\{i_1',i_2'\}$) means that we may assume that $\{i_1,i_2,i_1',i_2'\}$ does not contain triangle, and that the edges $\{i_1,i_2\}$ and $\{i_1',i_2'\}$ do not form a pair of triangles with some fifth vertex.  
In this case there are at most two edges between $\{i_1,i_2\}$ and $\{i_1',i_2'\}$.

Subcase 2.1.  \emph{There are two edges between $\{i_1,i_2\}$ and $\{i_1',i_2'\}$}.\\
 As there is no triangle, there is no loss of generality in assuming that the edges are $\{i_1,i_1'\}$ and $\{i_2,i_2'\}$.  But this is a $4$-cycle, and there are no such configurations in~$\mathbb{G}_I$.  So there is nothing to check.

Subcase 2.2. \emph{There is exactly one edge between $\{i_1,i_2\}$ and $\{i_1',i_2'\}$; say $\{i_1,i_1'\}$}.\\
By like colouring $i_2$ and $i_2'$ we have all three colours amongst $\{i_1,i_2,i_1',i_2'\}$.  So the goal is to achieve exactly $2$ colours amongst the four vertices.  
Let $i_3$ and $i_3'$ be vertices forming a triangle with $\{i_1,i_2\}$ and $\{i_1',i_2'\}$ respectively.  
Now $i_3$ cannot coincide with $i_3'$ as then $\{i_1,i_2,i_3\}$ would be a triangle sharing a common edge with the triangle $\{i_1,i_1',i_3\}$, and $\mathbb{G}_I$ has no such configurations.  
Also $i_3$ is not amongst $\{i_1',i_2'\}$, nor is $i_3'$ amongst $\{i_1,i_2\}$, because there is only one edge between the vertices of the two edges being considered in Subcase 2.2.  
Now $i_3$ and $i_3'$ cannot be adjacent because if so we would find the four cycle $i_3\sim i_3'\sim i_1'\sim i_1\sim i_3$, which does not appear in $I_3$ by Remark \ref{rem:4cycles}.  
Thus we may  like-colour $i_3$ and $i_3'$ to achieve exactly two colours amongst our four vertices. Thus Subcase 2.2 is complete.

Subcase 2.3.  \emph{There are no edges between $\{i_1,i_2\}$ and $\{i_1',i_2'\}$}.\\
  Again let $i_3$ and $i_3'$ be vertices forming triangles with $\{i_1,i_2\}$ and $\{i_1',i_2'\}$ respectively.  
  As there are no edges between $\{i_1,i_2\}$ and $\{i_1',i_2'\}$, it follows that  $i_3\notin\{i_1',i_2'\}$ and $i_3'\notin \{i_1,i_2\}$.  
  Also $i_3\neq i_3'$ {as otherwise $\{i_1,i_2,i_1',i_2',i_3=i_3'\}$ would be the first configuration in Figure \ref{fig:config}}, where we already know exactly two colours can be found amongst $\{i_1,i_2,i_1',i_2'\}$ under any valid $3$-colouring.  
  Now if $\{i_3,i_3'\}$ is an edge then we have the second configuration of Figure \ref{fig:config}, and have already observed that exactly $3$ colours appear amongst $\{i_1,i_2,i_1',i_2'\}$, under any valid $3$-colouring.   
  Thus we may now assume that $\{i_3,i_3'\}$ is not an edge.  But  then we can like-colour $i_3$ and $i_3'$  to place exactly two colours amongst $\{i_1,i_2,i_1',i_2'\}$, and unlike-colour $i_3$ and $i_3'$  to place exactly 3 colours amongst $\{i_1,i_2,i_1',i_2'\}$ as required.
\end{proof}
 
The next proposition uses the $\texttt{NP}$-hardness of distinguishing $k$-colourable graphs from $n$ colourable graphs, where $n>k\geq 3$.  At the time of writing, the  state of art is when $n=2k-1$; see Bul\'{\i}n, Krokhin and Opr\v{s}al \cite{BKO} or the same team with Barto \cite[Theorem 6.5]{BBKO}.
 \begin{pro}\label{pro:highercolouring}
 Let $k\geq 3$.  Then it is \texttt{NP}-hard to distinguish between the following two classes of graphs\up:
 \begin{itemize}
 \item graphs that are ${\leq}k$-robustly $(2k-1)$-colourable\up;
 \item graphs that are not $(2k-1)$-colourable.
 \end{itemize}
 \end{pro}
 \begin{proof}
 Let $\mathbb{G}$ be an instance of the promise: either $\mathbb{G}$ is $k$-colourable or is not $(2k-1)$-colourable.  We need to show that when $\mathbb{G}$ is $k$-colourable then it is $\leq k$-robustly $(2k-1)$-colourable.
 Let $p$ be any valid partial $2k-1$-colouring of at most $k$ vertices $v_0,\dots,v_{\ell-1}$.  We may assume that $\ell\geq 1$, as otherwise we may extend the empty domain of $p$ to include some vertex $v_0$ coloured arbitrarily.  Let $c$ be a valid $k$-colouring of $\mathbb{G}$.
 
Let $C\subseteq \{0,\dots,2k-1\}$ denote the set of colours used by $p$, and assume without loss of generality that $p(v_0)=0$.  If $|C|=k$ then $\ell=k$ and all of the vertices $v_0,\dots,v_{\ell-1}$ have different colours under $p$.  In this situation, there are $k-1$ remaining colours $\{0,\dots,2k-1\}\backslash C$ unused by $p$, and note that, up to a permutation of the colour names, we may assume that the $k$-colouring $c$ maps into $\{0\}\cup \big(\{0,\dots,2k-1\}\backslash C\big)$ in such a way that $c(v_0)=0$.  Now update $c$ on the colours $v_1,\dots,v_{\ell-1}$ by redefining $c$ to agree with $p$ on these points.  This is a valid $(2k-1)$-colouring that extends $p$.

Alternatively, if $|C|<k$, then there are at least $k$ colours unused by $p$, and we may assume that $c$ maps into $\{0,\dots,2k-1\}\backslash C$.  Again, we may update $c$ to agree with $p$ on the particular points $v_0,\dots,v_{\ell-1}$ to obtain a valid $(2k-1)$-colouring extending $p$.
 \end{proof}

 \subsection{Some comments.} 
 To finish this section, we make a brief comment about the prospect of following the direct approach to $2$-robust G3C sketched in \cite{AGK}.  {The approach in \cite{AGK}  starts from a result of Beacham \cite{bea} showing the \texttt{NP}-completeness of $2$-robust satisfiability for $3$-hypergraph $2$-colourability.  The $3$-hypergraph $2$-colourability problem is simply the positive version of NAE3SAT and the reduction to G3C in \cite{AGK} also follows a standard construction.  Unfortunately, not every edge  lies within a triangle, and while triangulation of edges by the addition of extra vertices obviously preserves the $3$-colourability or otherwise of a graph, it does not preserve the $2$-robust colourability.  We will require triangulated edges (or at least a condition equivalent to this) for our reduction in the next section.}

The author's original approach to $2$-robust $3$-colourability was independent of all of these possibilities, and instead arrived via a reduction from the Hamiltonian Circuit problem (via an intermediate encoding into SAT then NAE3SAT).  The idea is that if there is a Hamiltonian circuit in a graph, then we can also choose where to start this circuit, and in which direction to travel.  By varying these choices it is possible to again arrive at a reduction to G3C, in which yes instances are $2$-robustly satisfiable and no instances are not $3$-colourable at all.  However, the author was unable to push this to the results of the next section (due to essentially the same problem as for the \cite{AGK} approach) until a conversation with Standa \v{Z}vin\'y\ at the ALC Fest in Prague 2014 drew attention to the more flexible approach of Gottlob~\cite{got}, which we have revisited at the start of Section \ref{sec:sat}.
  
\section{$2$-robust positive 1-in-3SAT and separability.}\label{sec:monotone}
We now continue our series of reductions $I_0=I,I_1,I_2, I_3=\mathbb{G}_I$ by encoding the graph $\mathbb{G}_I$ into positive 1-in-3SAT; the main result is Theorem~\ref{thm:SP2}, which was the primary target of all previous sections in Part~\ref{part:sat}.  Recall from Section \ref{sec:CSPintro} that $\mathbbold{2}$ denotes the usual template for positive 1-in-3SAT as a constraint problem: the domain is the set $\{0,1\}$ and there is a single ternary relation $R=\{(0,0,1),(0,1,0),(1,0,0)\}$.

We define our construction starting from any graph $\mathbb{G}$, producing an instance~$\mathbbold{2}(\mathbb{G})$ of 1-in-3SAT.  The instance $I_4$ will be taken to be $\mathbbold{2}(I_3)$.

Construction of the positive 1-in-3SAT variables of $\mathbbold{2}(\mathbb{G})$ is straightforward: for each vertex $i$ of the graph $\mathbb{G}$ and each colour $j=0,1,2$, we construct the variable $x_{i,j}$ (which asserts ``vertex $i$ is coloured $j$'').

We now give the clauses of $\mathbbold{2}(\mathbb{G})$, which are of two kinds, \emph{type A2} and \emph{type B2}.   
\begin{itemize}
\item[A2] For each vertex $i$ we have the clause $(x_{i,0}\vee x_{i,1}\vee x_{i,2})$.  1-in-3 satisfaction of this clause corresponds to ``vertex $i$ is coloured by exactly one colour''.
\item[B2] For each triangle $\{i_1,i_2,i_3\}$ appearing in $\mathbb{G}$ and each colour $j=0,1,2$, we include the clause $(x_{i_1,j}\vee x_{i_2,j}\vee x_{i_3,j})$.  1-in-3 satisfaction of this clause corresponds to ``colour $j$ appears exactly once in the {triangle} $\{i_1,i_2,i_3\}$''.
\end{itemize}
In this construction, and for the remainder, we will not distinguish between clauses that differ only up to the order of appearance of their variables, so $(x_{i,0}\vee x_{i,1}\vee x_{i,2})$ is considered the same as $(x_{i,1}\vee x_{i,2}\vee x_{i,0})$ and so on.
In the case that every edge of $\mathbb{G}$ lies within a triangle, then it is easy to see that 1-in-3 satisfying truth assignments for the collection of type A2 and B2 clauses coincide precisely to valid $3$-colourings of $\mathbb{G}_I$ under the given interpretation of the variables.  {This gives the easy half of the next theorem (the first dot point).}
\begin{thm}\label{thm:1in3}
Let $\mathbb{G}$ be a graph in which every edge lies in a triangle. 
\begin{itemize}
\item If  $\mathbb{G}$ is \emph{not} $3$-colourable then $\mathbbold{2}(\mathbb{G})$ is \emph{not} 1-in-3 satisfiable.
\item If $\mathbb{G}$ is ${\leq} 2$-robustly satisfiable, then $\mathbbold{2}(\mathbb{G})$ is ${\leq} 2$-robustly 1-in-3 satisfiable.
\end{itemize}
\end{thm}
\begin{proof}
Only the second dot point remains to be established.  Assume $\mathbb{G}$ is ${\leq}2$-robustly $3$-colourable.
Let $x$ and $y$ be two distinct variables of $\mathbbold{2}(\mathbb{G})$; say $x=x_{i,j}$ and $y=x_{i',j'}$.  We must show that every locally compatible assignment of values to $x$ and $y$ extends to a 1-in-3 satisfying truth assignment.  

Case 1: $i=i'$.  In this case, local consistency ensures that not both $x$ and $y$ can take the value $1$ (due to clauses of type A2).  Without loss of generality, assume that $j=0$ and $j'=1$.  In this case, as $x\neq y$ we can easily range through the values $(0,0),(0,1),(1,0)$  for $(x,y)$ by colouring vertex $i$ by $2$, $1$ and $0$ respectively.  

Case 2:  $i\neq i'$.  As distinct vertices of $\mathbb{G}$ can unlike coloured, if $j\neq j'$ (say, $j=0$, $j'=1$) then we may achieve the following values of $(x,y)$ in the following ways:
\begin{itemize}
\item $(0,0)$ by colouring $i$ by $1$ and $i'$ by $0$.
\item $(0,1)$ by colouring $i$ by $2$ and $i'$ by $1$.
\item $(1,0)$ by colouring $i$ by $0$ and $i'$ by $2$.
\item $(1,1)$ by colouring $i$ by $0$ and $i'$ by $1$.
\end{itemize}
If $j=j'$ (say, $j=0$), then the possibilities depend on whether or not $\{i,i'\}$ is an edge.  If it is, then local consistency with clauses of kind B2 ensures we do not need to achieve $(x,y)$ taking the value $(1,1)$.  We may achieve the following values of $(x,y)$ in the following ways:
\begin{itemize}
\item $(0,0)$ by colouring $i$ by $1$ and $i'$ by $2$.
\item $(0,1)$ by colouring $i$ by $1$ and $i'$ by $0$.
\item $(1,0)$ by colouring $i$ by $0$ and $i'$ by $1$.
\item $(1,1)$ by colouring $i$ by $0$ and $i'$ by $0$, which is possible when $\{i,i'\}$ is \emph{not} an edge.  
\end{itemize}
This completes the proof.
\end{proof}
The instance $I_3=\mathbb{G}_I$ has every edge lying within a triangle and is either not $3$-colourable or is ${\leq} 2$-robustly $3$ colourable.  Thus the instance $I_4=\mathbbold{2}(I_3)$ satisfies the premise of precisely one of the two points in Theorem~\ref{thm:1in3}.
\begin{thm}\label{thm:SP2}
Let $\mathbbold{2}$ denote the usual template for positive 1-in-3SAT.  Any set of positive 3SAT instances containing the ${\leq}2$-robustly positive 1-in-3 satisfiable members of $\mathsf{SP}(\mathbbold{2})$ and contained within the members of $\CSP(\mathbbold{2})$ has \texttt{NP}-hard membership with respect to logspace many-one reductions.
\end{thm}
\begin{proof}
The chain of reductions $I\mapsto I_4$ maps (by a logspace many-one reduction) NO instances  of the \texttt{NP}-complete problem positive NAE3SAT
to NO instances of $\CSP(\mathbbold{2})$ and YES instances to the ${\leq}2$-robustly 1-in-3 satisfiable instances of $\mathbbold{2}$.  However when $I_4$ is ${{\leq} 2}$-robustly 1-in-3 satisfiable, it lies in $\mathsf{SP}(\mathbbold{2})$ because it satisfies the $\mathcal{R}\cup\{=\}$-separation properties for $\mathbbold{2}$, where $\mathcal{R}$ denotes the signature of $\mathbbold{2}$ (a single ternary relation).  
\end{proof}

In the next section we need a number of particular properties held by the instances of positive 1-in-3SAT that are obtained via this reduction; {as before, we express them in terms of $\mathbbold{2}(\mathbb{G})$ for a general graph $\mathbb{G}$}.  
Before we list these properties, let us say that two pairs of variables $\{v,w\}$ and $\{x,y\}$ are \emph{linked} in an instance of 1-in-3SAT if there is a variable $z$ such that $(v\vee w\vee z)$ and $(x\vee y\vee z)$ are clauses appearing in $\mathbbold{2}(\mathbb{G})$.  
(The reader is reminded that we consider clauses to be the same if they differ only up to the order of appearance of variables, so the common variable in a pair of linked clauses does not need to appear in the final position in general.)  {We define the \emph{link graph} of $\mathbbold{2}(\mathbb{G})$ to be the graph $\operatorname{link}(\mathbbold{2}(\mathbb{G}))$ whose vertices are the $2$-element sets of variables of $\mathbbold{2}(\mathbb{G})$ and where two such vertices are adjacent when they are linked.  
We will also say that a vertex $\{x,y\}$ of $\operatorname{link}(\mathbbold{2}(\mathbb{G}))$ is of \emph{type A} when $x,y$ form part of a clause of type A2, and are of \emph{type B} when $x,y$ are part of a clause of type B2.  
All other pairs of vertices are of \emph{type C}; note that type C  vertices are isolated in $\operatorname{link}(\mathbbold{2}(\mathbb{G}))$.}
In the following lemma, ``distinct clauses'' means clauses that differ by at least one variable.
\begin{lem}\label{lem:fourprops}
{Let $\mathbb{G}$ be a graph in which every edge lies within a triangle and there are no 4-cycles.}
{The following properties hold for $\mathbbold{2}(\mathbb{G})$.}
\begin{enumerate}[\rm(I)]
\item No two distinct clauses {of $\mathbbold{2}(\mathbb{G})$} share more than one variable.
\item {Let $x,y,z$ be variables of $\mathbbold{2}(\mathbb{G})$, not necessarily distinct and such that the clause $(x\vee y\vee z)$ does \emph{not} appear amongst the constructed clauses of $\mathbbold{2}(\mathbb{G})$.  Then either $|\{x,y,z\}|\leq 2$ or at least one of the pairs $\{x,y\}$, $\{x,z\}$, $\{y,z\}$ is of type C. }
\item Let $w,x,y,z$ be not necessarily distinct variables of $\mathbbold{2}(\mathbb{G})$.  {Then either $|\{w,x,y,z\}|\leq 3$ or there is a two element subset $\{u,v\}\subseteq\{w,x,y,z\}$ with $\{u,v\}$ of type C.}
\item {The link graph $\operatorname{link}(\mathbbold{2}(\mathbb{G}))$ is a disjoint union of cliques.  Moreover these cliques are either degenerate isolated vertices or cliques of size $2$ \up(that is, isolated edges\up).}
\end{enumerate}
\end{lem}   
\begin{proof}
Because $\mathbb{G}$ contains no $4$-cycles it follows that no two triangles of $\mathbb{G}$ share a common edge.  Property (I) follows immediately from this and the construction of~$\mathbbold{2}(\mathbb{G})$ from $\mathbb{G}$.

{For (II), assume that $(x\vee y\vee z)$ fails to appear in $\mathbbold{2}(\mathbb{G})$ and $|\{x,y,z\}|=3$.  If $\{x,y\}$ is of type C then we are done.  Now consider when 
 $\{x,y\}$ is of type $A$.  Then there is no loss of generality in assuming that $x=x_{i,0}$, $y=x_{i,1}$ and $z=x_{j,k}$, where $j\neq i$.  Now select $\ell\in\{0,1\}$ such that $k\neq \ell$, and find that $\{x_{i,\ell},z\}$ is of type C.  When $\{x,y\}$ is of type B there is no loss of generality in assuming that $x=x_{i_1,0}$, $y=x_{i_2,0}$ and $z=x_{j,k}$ where either $k\neq 0$ or $\{i_1,i_2,j\}$ is a triple of pairwise distinct vertices of $\mathbb{G}$ that do not form a triangle.  If $k\neq 0$ then there is a choice of $i\in \{i_1,i_2\}$ such that $j\neq i$.  Then $\{x_{i,0},z\}$ is of type C.  \ If $k=0$ but $\{i_1,i_2,j\}$ does not form a triangle, then given that $\{i_1,i_2\}$ is an edge of $\mathbb{G}$  we must have that one of $\{i_1,j\}$ or $\{i_2,j\}$ is not an edge of $\mathbb{G}$.  Without loss of generality assume $\{i_1,j\}$ is not an edge.  Then $\{x_{i_1,0},x_{j,k}\}$ is of type C.  }
 
{Observation (III) will follow immediately from Observation (II) unless both $(w\vee x\vee y)$ and $(x\vee y\vee z)$ are clauses in $I_4$.  But then Observation (I) shows that these clauses must be identical (up to a permutation of the position of variables in the clause), so that $w=z$ and we have $|\{w,x,y,z\}|\leq 3$.}

{For  (IV), assume that $\{y_1,z_1\}$ links to $\{y_2,z_2\}$ which links to $\{y_3,z_3\}$.  
So there are variables $u,v$ such that $(y_1\vee z_1\vee u)$, $(y_2\vee z_2\vee u)$ as well as $(y_2\vee z_2\vee v)$ and $(y_3\vee z_3\vee v)$ are clauses of $\mathbbold{2}(\mathbb{G})$, up to a permutation of the position of variables in the clause.  However, $(y_2\vee z_2\vee u)$ and $(y_2\vee z_2\vee v)$ share more than one variable (this uses the fact that every clause of $\mathbbold{2}(\mathbb{G})$ contains three variables), so must coincide by~(I).  
Hence  $u=v$ showing that $\{y_1,z_1\}$ links to $\{y_3,z_3\}$.  
Now observe from the definition of clauses of type A2 that there can be no edge between two vertices of type A, while the fact that no triangle shares a common edge similarly ensures that there can be no edge between vertices of type B.  
Thus all edges in $\operatorname{link}(\mathbbold{2}(\mathbb{G}))$ are between type A and B, which implies it is bipartite.  Because every nontrivial component is a clique it then follows that all nontrivial components of $\operatorname{link}(\mathbbold{2}(\mathbb{G}))$ are single edges.}
\end{proof}

\part{Applications to variety membership}\label{part:varieties}
\section{Variety membership for the semigroup ${\bf B}_2^1$}\label{sec:B21}
Recall that the variety generated by an algebraic structure $\mathbf{A}$ is the class of all algebras of the same signature that satisfy the equations holding on $\mathbf{A}$; equivalently that are homomorphic images of subalgebras of direct powers of $\mathbf{A}$.  {When $\mathbf{A}$ is finite}, the pseudovariety of $\mathbf{A}$ is the class of finite algebras in the variety of $\mathbf{A}$, but more generally a pseudovariety is a class of finite algebras of the same signature that is closed under taking homomorphisms, subalgebras and finitary direct products.  The finite part of any variety is a pseudovariety, but not every pseudovariety is of this form.  The computational problem of deciding membership of finite algebras in a variety coincides with the problem of deciding membership in the pseudovariety of finite members.  For this reason, the results of the present section are cast in the setting of pseudovarieties.

The six element monoid ${\bf B}_2^1=\langle 1,a,b\mid a^2=b^2=0, aba=a, bab=b\rangle$ is the most ubiquitous minimal counterexample in problems relation to varieties of semigroups.
The possible intractability of the computational problem of deciding membership of finite semigroups in the variety generated by ${\bf B}_2^1$ is perhaps the most obvious unresolved problem relating to ${\bf B}_2^1$ and so it is not surprising that this has appeared in a number of places in the literature including Problem 4 of Almeida \cite[p.~441]{alm}, Problem 3.11 of Kharlampovich and Sapir \cite{khasap} and page 849 of Volkov, Gol$'$dberg and Kublanovksi\u{\i} \cite{VGK}.  We now use the results of the previous section to show that this computational problem is \texttt{NP}-hard.   In fact we show this true for a large swathe of finite semigroups: any semigroup whose pseudovariety contains that of ${\bf B}_2^1$ and lies within the join of the pseudovariety of ${\bf B}_2^1$ with the pseudovariety $\mathsf{L}{\bf DS}$ {(here the font variation follows standard convention due to the fact that $\mathsf{L}$ is a class operator applied to the smaller pseudovariety ${\bf DS}$; see Almeida \cite{alm} for example).}

We build on the reduction to $2$-robust positive 1-in-3SAT in the previous section, {continuing with $I_4$ as our input structure for the next reduction}.  The notation $V_{I_4}$ will denote the set of variables appearing in $I_4$.  Recall that we consider clauses only up to the rearrangement of the variables they contain.  

We construct a semigroup ${\bf S}_{I_4}$ by way of the following semigroup presentation.  The semigroup ${\bf S}_{I_4}$ is generated from the following elements: 
\begin{enumerate}
\item $1$, a multiplicative identity element;
\item $0$, a multiplicative $0$ element;
\item an element $\ba$ and an element $\bb$;
\item for each variable $v$ of $I$ an element $\ba_v$.
\end{enumerate}
The semigroup is subject to the following rules (where $u,v,u_1,u_2,v_1,v_2,w$ are variables in the instance $I_4$):
\begin{enumerate}
\item ${\bf a}{\bf b}{\bf a}={\bf a}$, ${\bf b}{\bf a}{\bf b}={\bf b}$, ${\bf a}{\bf a}={\bf b}{\bf b}=0$ (so that $1, {\bf a},{\bf b}$ generate a copy of ${\bf B}_2^1$);
\item $\ba_u\ba_v=0$ if there is no clause containing both $u$ and $v$ (or if $u=v$);
\item $\ba_u\ba_v=\ba_v\ba_u$;
\item $\ba_u\ba_v\ba_w=\ba$ whenever $u\vee v\vee w$ is a clause in $I_4$;
\item $\ba_{u_1}\ba_{u_2}=\ba_{v_1}\ba_{v_2}$ if $\{u_1,u_2\}$ and $\{v_1,v_2\}$ are linked pairs of variables. 
\item {$\bb\ba_{u}\ba_{v}\bb=\bb\ba_u\bb=0$.}
\end{enumerate}
Observe that conditions (4) and (6) show that $\ba$ and $0$ are redundant as generators, and ${\bf S}_{I_4}$ can be generated from $1,\bb,\ba_v$ (for variables $v$) alone.

When $u$ and $v$ appear in some common clause, then we let $[uv]$ denote the set of all words $xy$ where $x$ and $y$ are variables of $I_4$ and $\{u,v\}$ is linked to $\{x,y\}$; so $[uv]=[xy]=[yx]$ in this case.   
{We will write $\ba_{[uv]}$ to denote the value $\ba_u\ba_v$ and observe that this notation is consistent with condition (5), which gives  $\ba_{[uv]}=\ba_u\ba_v=\ba_x\ba_y=\ba_{[xy]}$ when $\{u,v\}$ is linked to $\{x,y\}$.}

We considered clauses $u\vee v\vee w$ only up to permutations of the variables within them, which is consistent with the commutativity condition (3) and implies that condition (5) also holds for all permutations of the product $\ba_u\ba_v\ba_w$.
Also, for any variable $u$ we have $\ba_u\ba_u=0$ by (2), because of the property that the variable $u$ does not appear twice in any single clause.    
Finally, we observe that for any quadruple of variables $v_1,v_2,v_3,v_4$ (not necessarily distinct) we have $\ba_{v_1}\ba_{v_2}\ba_{v_3}\ba_{v_4}=0$.  
Indeed, if $|\{v_1,v_2,v_3,v_4\}|\leq 3$ then (2) and (3) yield $0$ as the value of the product, while if  $|\{v_1,v_2,v_3,v_4\}|=4$, condition (III) from Lemma~\ref{lem:fourprops} shows that there is $2$-element subset of $\{v_1,v_2,v_3,v_4\}$ of type C, from which (2) and (3) again imply the value of~$0$.

Thus the elements of ${\bf S}_{I_4}$ are $A\cup A\bb A\cup\{0\}$, where 
\[
A:=\{1,\ba\}\cup\{\ba_u\mid u\in V_{I_4}\}\cup \{\ba_{[uv]}\mid u,v\in V_{I_4}\text{ appear in a single clause}\}.
\]
We mention that the equality in condition (5) of the presentation does not lead to further equalities between the elements just listed: this is due to property (IV) of Lemma \ref{lem:fourprops} and the equality (4) of the presentation.
In particular the semigroup~${\bf S}_{I_4}$ can be constructed in logspace from the instance $I_4$, given that the instance $I_4$ satisfies the particular properties identified above (itself constructible in logspace from $I_3$, then from $I_2$, then from $I_1$ and finally  from $I$).   

{In the following lemma we use $\mathsf{V}_{\rm fin}({\bf B}_2^1)$ to denote the pseudovariety generated by ${\bf B}_2^1$ (it consists of the finite members of the variety generated by ${\bf B}_2^1$).}
\begin{lem}\label{lem:isinB21}
If $I_4$ is 1-in-3 satisfiable, then ${\bf S}_{I_4}\in\mathsf{V}_{\rm fin}({\bf B}_2^1)$.
\end{lem}
\begin{proof}
If $I_4$ is 1-in-3 satisfiable then it is ${\leq}2$-robustly 1-in-3 satisfiable by Theorem~\ref{thm:1in3}.
Let ${\bf T}$ be the subsemigroup of $({\bf B}_2^1)^{\hom(I_4,\mathbbold{2})}$ generated by the maps 
\[
\beta,\alpha,\alpha_{v},1\colon\hom(I_4,\mathbbold{2})\to {\bf B}_2^1
\] 
(for each $v\in V_{I_4}$) defined by their action on $\phi\in \hom(I_4,\mathbbold{2})$ as follows:
\begin{align*}
\beta(\phi)&:=b;\\
\alpha(\phi)&:=a;\\
\alpha_v(\phi)&:=
\begin{cases}
1&\text{ if $\phi(v)=0$,}\\
a&\text{ if $\phi(v)= 1$;}
\end{cases}\\
1(\phi)&:=1.
\end{align*}
Let $U\subseteq T$ (the underlying universe of ${\bf T}$) consist of all elements $t$  of ${\bf T}$ that contain a $0$-entry; that is, there is $\phi\in \hom(I_4,\mathbbold{2})$ with $t(\phi)=0$.  The set $U$ is trivially seen to be an ideal of ${\bf T}$, and we claim that  ${\bf S}_{I_4}$ is isomorphic to the Rees quotient ${\bf T}/U$ under the map $\iota$ defined by $\ba_v\mapsto \alpha_{v}$, ${\bf b}\mapsto \beta$, $\ba\mapsto \alpha$ and $1\mapsto 1$.  
We first show that~$\iota$ is bijective on the generators, or equivalently, that there are no equalities between the elements $1$, $\beta$, $\alpha$ and the various $\alpha_v$.  This is trivially true between $1$, $\beta$ and $\alpha$.  For each variable $v$ in $V_{I_4}$, the ${\leq} 2$-robust satisfiability of $I_4$ ensures there are homomorphisms $\nu_1,\nu_2:I_4\to \mathbbold{2}$ with $\nu_1(v)=1$ and $\nu_2(v)=0$.  Then $\alpha_v(\nu_1)=a$ and $\alpha_v(\nu_2)=1$.  So $\alpha_v\notin\{1,\alpha,\beta\}$.  Similarly, for distinct variables $u,v$ it is locally compatible to satisfy $I_4$ by a homomorphism $\nu$ with $\nu(u)\neq \nu(v)$, and then $\alpha_u(\nu)\neq \alpha_v(\nu)$.  This shows that $\iota$ maps the generators of ${\bf S}_{I_4}$ bijectively to those of ${\bf T}/U$.

We now verify that the rules (1)--(6) defining ${\bf S}_{I_4}$ hold for the corresponding elements of ${\bf T}/U$, which will show that $\iota$ is a surjective homomorphism.  It will then remain to show that elements distinct in ${\bf S}_{I_4}$ correspond to distinct elements of ${\bf T}/U$.

Rules  (1) and (3) are trivial to verify.  For rule (2), observe that when $u,v$ are variables  not both appearing in the same clause, then the partial assignment $(u,v)\mapsto (1,1)$ is locally compatible, so there is $\nu\in\hom(I_4,\mathbbold{2})$ with $\nu(u)=\nu(v)=1$.  Then on coordinate $\nu$ we have both $\alpha_u(\nu)=a$ and $\alpha_v(\nu)=a$ so that $[\alpha_u\alpha_v](\nu)=0$, showing that $\alpha_u\alpha_v\in U$ (or is $0$ in ${\bf T}/U$). For rule~(4), note that the definition of 1-in-3 satisfaction ensures that exactly one variable in each clause $(x\vee y\vee z)$ takes the value $1$, which in turn implies that exactly one of $\alpha_x(\nu),\alpha_y(\nu),\alpha_z(\nu)$ takes the value $a$, with the others taking the value $1$.  Thus $[\alpha_x\alpha_y\alpha_z](\nu)=a$ (for all $\nu\in\hom(I,\mathbbold{2})$), showing that $\alpha_x\alpha_y\alpha_z=\alpha$ when $(x\vee y\vee z)$ is a clause in $I_4$.

For the rule (5), consider $\{u_1,u_2\}$ and $\{v_1,v_2\}$ such that $(w\vee u_1\vee u_2)$ and $(w\vee v_1\vee v_2)$ are clauses in $I_4$.  Let $\nu\in\hom(I_4,\mathbbold{2})$ be arbitrary.  
If $\nu(w)=1$, then $\nu(u_1)=\nu(u_2)=\nu(v_1)=\nu(v_2)=0$ and so $\alpha_{u_1}\alpha_{u_2}(\nu)=\alpha_{v_1}\alpha_{v_2}(\nu)=1$.  
If $\nu(w)=0$, then $\{\nu(u_1),\nu(u_2)\}=\{\nu(v_1),\nu(v_2)\}=\{0,1\}$ and so $\alpha_{u_1}\alpha_{u_2}(\nu)=\alpha_{v_1}\alpha_{v_2}(\nu)=a$.  
As $\nu\in\hom(I_4,\mathbbold{2})$ is arbitrary, we have $\alpha_{u_1}\alpha_{u_2}=\alpha_{v_1}\alpha_{v_2}$.  
As for elements of ${\bf S}_{I_4}$, we now write $\alpha_{[uv]}$ to denote the product $\alpha_u\alpha_v$.

For the rule (6), consider any pair $\{u,v\}$.  If $\{u,v\}$ is of type C then $\alpha_u\alpha_v=0$ in ${\bf T}/U$, so we now assume it is of type A or B. \  By ${\leq} 2$-robust satisfiability there is $\nu\in\hom(I_4,\mathbbold{2})$ such that $\nu(u)=\nu(v)=0$, and then $\alpha_{[uv]}(\nu)=1$.  So in ${\bf T}$ we have $\beta\alpha_{[uv]}\beta(\nu)=b 1 b=0$, showing that $\beta\alpha_{[uv]}\beta\in U$.  Thus in ${\bf T}/U$ we have $\beta\alpha_{[uv]}\beta=0$.  The case of $\beta\alpha_u\beta$ is almost identical, using the $1$-robust satisfiability of $I_4$ to find a $\nu\in\hom(I_4,\mathbbold{2})$ with $\nu(u)=0$.

This shows that the map $\iota:{\bf S}_{I_4}\to {\bf T}/U$ is a surjective homomorphism.  To complete the proof we must show that $\iota$ is injective.
This has already been verified for the generators.
Observe that each element $\alpha_u$ has a coordinate equal to $1$.  Similarly, if $u,v$ appear in some common clause of $I_4$, then, as $(u,v)\mapsto (0,0)$ is locally compatible, there is a coordinate $\nu\in \hom(I,\mathbbold(2))$ with $\alpha_u\alpha_v(\nu)=1$.  Thus elements of the form $\alpha_u$ or $\alpha_u\alpha_v$ (where $u,v$ appear in some clause of $I$) are distinct from any product involving  $\alpha,\beta$.

We have already verified that if $u\neq v$ then $\alpha_u\neq \alpha_v$.  Observe now that if $u$ and $v$ are distinct variables appearing in some clause of $I_4$ and $w$ is any variable, then without loss of generality we have $w\neq u$, so that $(w,u)\mapsto (0,1)$ is locally compatible.  Hence there is a coordinate $\nu\in \hom(I_4,\mathbbold(2))$ in which $\alpha_u\alpha_v(\nu)=a\neq 1=\alpha_w(\nu)$.  Hence no $\alpha_w$ is equal to any $\alpha_u\alpha_v$.

Next, consider when $u,v$ appear in some clause and $x,y$ appear in some clause but $\{u,v\}$ and $\{x,y\}$ are not linked.  Let $w$ be such that $(w\vee u\vee v)$ is a clause in $I$.  Now $(w\vee x\vee y)$ is not a clause, so by Lemma \ref{lem:fourprops} (II) {we have a pair from $\{w,x,y\}$ that has type C.  As $\{x,y\}$ is not of type C there is no loss of generality in assuming that $\{w,x\}$ is of type C.  The assignment $\nu$ given by $(\nu(w),\nu(x))=(1,1)$ is locally compatible so extends to a 1-in-3 satisfying assignment $\nu\in \hom(I_4,\mathbbold{2})$ that fails to 1-in-3 satisfy $(w\vee x\vee y)$.  Because $x,y$ and $u,v$ each  appear in some clause of $I_4$ it follows that
$(\nu(w),\nu(x),\nu(y))=(1,1,0)$, while $(\nu(w),\nu(u),\nu(v))=(1,0,0)$.} Then $\alpha_{[uv]}(\nu)=1$ while $\alpha_{[xy]}(\nu)=a$.  Thus $\alpha_{[uv]}\neq \alpha_{[xy]}$ if $[uv]\neq [xy]$.

We can now verify that there are no further equalities amongst the elements of ${\bf T}/U$ aside than for corresponding elements of ${\bf S}_{I_4}$.  Set $\overline{A}$ to be the analogue of the set $A$ in ${\bf T}/U$:
\[
\overline{A}:=\{1,\alpha\}\cup\{\alpha_u\mid u\in V_{I_4}\}\cup \{\alpha_{[uv]}\mid u,v\in V_{I_4}\text{ appear in a single clause}\}.
\]
We have shown that for any distinct $\gamma,\delta\in \overline{A}$ there is a 1-in-3 satisfying assignment $\nu_{\gamma,\delta}\in \hom(I_4,\mathbbold{2})$ such that $\{\gamma(\nu_{\gamma,\delta}),\delta(\nu_{\gamma,\delta})\}=\{1,a\}$; so $\iota$ is injective on the set~$A$.  
Consider then any $\gamma_1,\delta_1,\gamma_2,\delta_2\in \overline{A}$ such that $\gamma_1\beta\gamma_2=\delta_1\beta\delta_2\notin U$.  
For any $\nu\in \hom(I_4,\mathbbold{2})$ we have $\gamma_1\beta\gamma_2(\nu)\in \{b,ba\}$ if $\gamma_1(\nu)=1$ and $\gamma_1\beta\gamma_2(\nu)\in \{a,ab\}$ if $\gamma_1(\nu)=a$.  
Thus $\gamma_1\beta\gamma_2=\delta_1\beta\delta_2\notin U$ implies $\gamma_1=\delta_1$.  
By symmetry, $\gamma_2=\delta_2$ also, thus $\iota$ is  injective on elements of $A\bb A$ also.  Moreover, as all elements of $A$ have coordinates within $\{1,a\}$, and every element of $A\bb A$ has at least one coordinate in $\{b,ab,ba\}$, it follows that $\iota$ is injective, and therefore is an isomorphism.
\end{proof}

The converse to Lemma \ref{lem:isinB21} is true and implies that deciding membership in the pseudovariety of ${\bf B}_2^1$ is \texttt{NP}-hard.  We {will} show something stronger than the converse {however}.

The proof of the next result makes use of Green's relations $\mathscr{H}$, $\mathscr{L}$, $\mathscr{R}$, $\mathscr{D}$, and $\mathscr{J}$.  We briefly recall the definitions, but direct the reader to any text on semigroup theory  for further details; Howie \cite{how} is one such.
Let ${\bf S}$ be a semigroup, with underlying universe $S$, and let ${\bf S}^1$ denote the result of adjoining an identity element to ${\bf S}$ if it does not already have one (otherwise ${\bf S}^1={\bf S}$).  
Define the preorder $\leq_\mathscr{L}$ on $S$ by $a\leq_\mathscr{L} b$ if there exists $x\in S^1$ with $xb=a$.  
The preorder $\leq_\mathscr{R}$ is defined dually, while $\leq_\mathscr{H}$ is the intersection $\leq_\mathscr{L}\cap\leq_\mathscr{R}$.  
The preorder $\mathscr{J}$ is defined by $a\leq_\mathscr{J}b$ if there are $x,y\in S^1$ with $xby=a$.  The relations $\mathscr{L}$, $\mathscr{R}$, $\mathscr{H}$, $\mathscr{J}$ are the equivalence relations of elements equivalent under each respective preorder.  
Thus, for example $a\mathrel{\mathscr{J}}b$ if $a\leq_\mathscr{J}b$ and $b\leq_\mathscr{J}a$. 
A further useful fact that we use is that if $a\leq_\mathscr{H} b$ and $b\leq_\mathscr{J} a$ in a finite semigroup, then $a\mathrel{\mathscr{H}}b$; see \cite{how} for example.
Finally, we observe that $\mathscr{D}=\mathscr{L}\circ\mathscr{R}=\mathscr{R}\circ\mathscr{L}$ (the second equality is a theorem, not an obvious fact), and in a finite semigroup $\mathscr{D}=\mathscr{J}$.
Green's relations are one of the fundamental tools for the structural theory of semigroups, and often appear in characterisations of pseudovarieties.
The reader is directed to Almeida's text~\cite{alm} for the following pseudovariety facts.  
The pseudovariety $\mathbf{DS}$ consists of all finite semigroups whose regular $\mathscr{D}$-classes are subsemigroups, while   $\mathsf{L}\mathbf{DS}$ consists of those finite semigroups~${\bf S}$ that are \emph{locally} in $\mathbf{DS}$: for every idempotent $e$, the subsemigroup of ${\bf S}$ on elements $eSe$ is in $\mathbf{DS}$.  
The pseudovariety $\mathsf{L}\mathbf{DS}$ is known to capture precisely the class of all finite semigroups whose variety does not contain~${\bf B}_2^1$.
The interval $[\![\mathsf{V}_{\rm fin}({\bf B}_2^1),\mathsf{L}\mathbf{DS}\vee \mathsf{V}_{\rm fin}({\bf B}_2^1)]\!]$ consists of all pseudovarieties $\mathcal{V}$ with $\mathsf{V}_{\rm fin}({\bf B}_2^1)\subseteq \mathcal{V}\subseteq \mathsf{L}\mathbf{DS}\vee \mathsf{V}_{\rm fin}({\bf B}_2^1)$, or equivalently all pseudovarieties $\mathcal{V}\subseteq \mathsf{L}\mathbf{DS}\vee \mathsf{V}_{\rm fin}({\bf B}_2^1)$ containing~${\bf B}_2^1$.

\begin{lem}\label{lem:is3col}
Let $\mathcal{V}$ be any pseudovariety in the interval $[\![\mathsf{V}_{\rm fin}({\bf B}_2^1),\mathsf{L}\mathbf{DS}\vee \mathsf{V}_{\rm fin}({\bf B}_2^1)]\!]$.  If ${\bf S}_{I_4}\in\mathcal{V}$ then $I_4$ is 1-in-3 satisfiable.
\end{lem}
\begin{proof}
The proof works for any class $\mathcal{V}$ with $\mathsf{V}_{\rm fin}({\bf B}_2^1)\subseteq \mathcal{V}\subseteq \mathsf{L}\mathbf{DS}\vee \mathsf{V}_{\rm fin}({\bf B}_2^1)$ (not necessarily a pseudovariety).  
As ${\bf S}_{I_4}\in \mathcal{V}$, we also have ${\bf S}_{I_4}\in \mathsf{L}\mathbf{DS}\vee \mathsf{V}_{\rm fin}({\bf B}_2^1)$, which is equal to the closure of $\mathsf{L}\mathbf{DS}\cup \{{\bf B}_2^1\}$ under taking homomorphic images of subsemigroups of finitary direct products.  
So there is $\ell\geq 0$ and finite semigroups ${\bf S}_1,\dots,{\bf S}_k$ from $\mathsf{L}\mathbf{DS}$ such that there is ${\bf T}\leq {\bf S}_1\times\dots\times {\bf S}_k\times({\bf B}_2^1)^\ell$ and a surjective homomorphism $\psi:{\bf T}\onto {\bf S}_{I_4}$.  
As $\mathsf{L}\mathbf{DS}$ is closed under finitary direct products, there is no loss of generality in assuming that $k\leq 1$.  
Also, clearly $\ell>0$ because ${\bf S}_{I_4}\notin \mathsf{L}\mathbf{DS}$ due to the fact ${\bf S}_{I_4}$ embeds ${\bf B}_2^1$. 
For each variable $v$ appearing in $I$, select any element $\hat{a}_{v}\in\psi^{-1}({\bf a}_{v})$.  
Select an arbitrary element $\he$ from amongst the idempotents of the finite semigroup $\psi^{-1}(1)$ (unless one wishes to perform this proof in the monoid signature, and then the canonical choice of $\he$ is the constant tuple equal to $1$ on each projection). 
As ${\bf S}_{I_4}$ is generated by the elements ${\bf a}_{v}$, ${\bf b}$ and~$1$, we can assume that ${\bf T}$ is generated by the elements $\hat{a}_{v}$ along with $\he$ and any chosen element from $\psi^{-1}(\bb)$.  
We make a more careful selection of an element $
\hb$ from $\psi^{-1}(\bb)$.  

We select $\hat{b}$ arbitrarily from amongst the minimal elements of $\psi^{-1}(\bb)$ with respect to the $\mathscr{J}$ order.  Recall that for any clause $(u\vee v\vee w)$ in $I$ we have $\bb\ba_u\ba_v\ba_w\bb=\bb$, so that $\hb\ha_u\ha_v\ha_w\hb\in\psi^{-1}(\bb)$.  
Now $\hb\ha_u\ha_v\ha_w\hb\leq_\mathscr{H}\hb$, so we also have $\hb\ha_u\ha_v\ha_w\hb\leq_\mathscr{J}\hb$.  
Then the $\mathscr{J}$-minimality of the choice of $\hb$ in $\psi^{-1}(\bb)$ and the fact that $\hb\ha_u\ha_v\ha_w\hb\in\psi^{-1}(\bb)$, implies that we have $\hb\ha_u\ha_v\ha_w\hb\mathrel{\mathscr{J}}\hb$.  
But $\hb\ha_u\ha_v\ha_w\hb\leq_\mathscr{H}\hb$ then implies that $\hb\ha_u\ha_v\ha_w\hb\mathrel{\mathscr{H}}\hb$.  As this is true in ${\bf T}$, it is also true of any projection $\pi$ from ${\bf T}$ into ${\bf B}_2^1$.  As ${\bf B}_2^1$ is $\mathscr{H}$-trivial, it follows that $\hb\pi=\hb\ha_u\ha_v\ha_w\hb\pi$ always.  

Let $u\vee v\vee w$ be an arbitrary clause in $I_4$ and let $\ha$ denote the value $\ha_u\ha_v\ha_w$ (the value of $\ha$ may depend on the clause selected).  We define ${\bf T}'\leq {\bf T}$ to be the subsemigroup generated by $\ha,\hb,\he$, and note that $\psi$ maps ${\bf T}'$ homomorphically onto the subsemigroup of ${\bf S}_{I_4}$ generated by  $1,\bb,\ba$.  This subsemigroup of ${\bf S}_{I_4}$ is of course isomorphic to ${\bf B}_2^1$, showing that ${\bf B}_2^1 \in \mathsf{V}_{\rm fin}({\bf T}')$.  Now ${\bf T}'\leq {\bf S}_1\times ({\bf B}_2^1)^\ell$ so there are $\ell$ projection maps $\pi_1,\dots,\pi_\ell:T'\to {\bf B}_2^1$.
We wish to now show that one of these projections maps  ${\bf T}'$ surjectively onto ${\bf B}_2^1$.  Assume, for contradiction, that none of the projections $\pi_1,\dots,\pi_\ell$ map ${\bf T}'$ onto ${\bf B}_2^1$.  Then we may replace each of the $\ell$ copies of ${\bf B}_2^1$ in  ${\bf T}'\leq {\bf S}_1\times ({\bf B}_2^1)^\ell$ by a proper subsemigroup of ${\bf B}_2^1$.  But all proper subsemigroups of ${\bf B}_2^1$ lie in $\mathsf{L}{\bf DS}$ and this would show that ${\bf T}'\in \mathsf{L}{\bf DS}$, contradicting ${\bf B}_2^1\in\mathsf{V}_{\rm fin}({\bf T}')$.  Thus there is $\pi_i$ mapping ${\bf T}'$ surjectively onto ${\bf B}_2^1$.

Now ${\bf T}'$ is generated by $\he,\ha,\hb$ and so ${\bf B}_2^1$  is generated by $\he\pi_i$, $\ha\pi_i$, $\hb\pi_i$.  However there is only one $3$-element generating set for ${\bf B}_2^1$, and that is $\{1,a,b\}$.  As $\he$ is idempotent, while $a,b$ square to $0$, it follows that $\he\pi_i=1$, while $\{\ha\pi_i,\hb\pi_i\}=\{a,b\}$.  There is an automorphism of ${\bf B}_2^1$ that switches $a$ and $b$, so there is no loss of generality in assuming that  $\hb\pi_i=b$ and $\ha\pi_i=a$.

Now, for every variable $u$, there are variables $v,w$ such that $(u\vee v\vee w)$ is a clause in $I_4$, and then the property
 \[
[\hb\ha_u\ha_v\ha_w\hb]\pi_i = [\hb\ha_v\ha_w\ha_u\hb]\pi_i = [\hb\ha_w\ha_u\ha_v\hb]\pi_i=\hb\pi_i = b
 \]
 on projection $\pi_i$ gives $\ha_u\pi_i\in\{a,1\}$.  Define a truth assignment by $u\mapsto 1$ if $\ha_u\pi_i=a$ and $u\mapsto 0$ if $\ha_u\pi_i=1$.  Then for each clause $(u\vee v\vee w)$ in $I_4$, the property $[\hb\ha_u\ha_v\ha_w\hb]\pi_i=b$ ensures that exactly one of $\ha_u\pi,\ha_v\pi,\ha_w\pi$ is $a$, or equivalently, exactly one of $u,v,w$ is $1$.  Thus the clauses are satisfied in the 1-in-3 sense, as required.
 \end{proof}
 
 \begin{thm}\label{thm:LDS}
 Let $\mathcal{V}$ be any pseudovariety in the interval $[\![\mathsf{V}_{\rm fin}({\bf B}_2^1),\mathsf{L}\mathbf{DS}\vee \mathsf{V}_{\rm fin}({\bf B}_2^1)]\!]$.  Then membership in $\mathcal{V}$ is \texttt{NP}-hard with respect to logspace reductions.
 \end{thm}
 \begin{proof}
 We again continue with our instance $I_4$ and the constructed semigroup~${\bf S}_{I_4}$.
 If $I_4$ is a ${\leq} 2$-robustly 1-in-3 satisfiable instance of positive 1-in-3SAT, then ${\bf S}_{I_4}\in \mathsf{V}_{\rm fin}({\bf B}_2^1)\subseteq \mathcal{V}$ by Lemma \ref{lem:isinB21}.  Conversely, if ${\bf S}_{I_4}\in \mathcal{V}\subseteq \mathsf{L}\mathbf{DS}\vee \mathsf{V}_{\rm fin}({\bf B}_2^1)$ then Lemma~\ref{lem:is3col} shows that  $I_4$ is 1-in-3 satisfiable.  The result now follows from Theorem~\ref{thm:SP2}.
 \end{proof}

We now recall some basic facts about completely $0$-simple semigroups, and the reader is directed to a text such as Howie~\cite{how} for a more complete treatment.
Recall that a \emph{$0$-simple semigroup} is a semigroup ${\bf S}$ with~$0$ whose ideal consist only of $\{0\}$ and ${\bf S}$ itself. 
A $0$-simple semigroup is \emph{completely $0$-simple} if it satisfies the additional properties on nonzero idempotents: if $e^2=e$, $f^2=f$ and $ef=e=fe$, then $e=f$.
A \emph{completely simple semigroup} is defined similarly, but without~$0$: a single ideal, with the same condition on idempotents.
Every finite $0$-simple semigroup is completely $0$-simple, and moreover the seminal work of Rees~\cite{ree} provides an elegant representation of completely $0$-simple semigroups in terms of what is now called the Rees matrix semigroup (with $0$) construction, which we now recall.
Let ${\bf G}$ be a group with underlying set $G$, and $P$ an $I\times J$ matrix (the ``sandwich matrix'') whose entries come from $G\cup\{0\}$, where~$0$ is a symbol not in $G$.  
Then the \emph{Rees matrix semigroup} $M^0[{\bf G},P]$ consists of $0$ along with all triples $(i,g,j)$ where $i\in I$, $j\in J$ and $g\in G$.  The element $0$ is a multiplicative zero, and other products are defined by 
\[
(i,g,j)\cdot (k,h,\ell)=\begin{cases}
(i,gP_{jk}h,\ell)&\text{ if $P_{jk}\in G$}\\
0&\text{ otherwise.}
\end{cases}
\]
Provided that every row and every column of $P$ contains a nonzero element, the semigroup $M^0[{\bf G},P]$ is completely $0$-simple.  
Remarkably, every completely $0$-simple semigroup is isomorphic to a Rees matrix semigroup of this form.  
The semigroup~${\bf B}_2$ for example, is isomorphic to $M^0[{\bf 1},I_2]$, where ${\bf 1}$ denotes the one-element group, and $I_2$ is the $2\times 2$ identity matrix, with the single element of ${\bf 1}$ on the diagonal.
In the following results we always consider completely $0$-simple semigroups as Rees matrix semigroups.

The following observation is not found in \cite{how}, so we give more details.  
Consider two Rees matrix semigroups $M^0[{\bf G},P]$ and  $M^0[{\bf H},Q]$ (where $P$ has dimensions $I_P\times J_P$ and $Q$ dimensions $I_Q\times J_Q$).  
If we factor the direct product $M^0[{\bf G},P]\times M^0[{\bf H},Q]$ by the maximal ideal consisting of all tuples that include at least one $0$-coordinate, then a completely $0$-simple semigroup results that is isomorphic to $M^0[{\bf G}\times {\bf H},P\times Q]$, where $P\times Q$ denotes a form of the matrix Kronecker product: its dimensions are $(I_P\times I_Q)\times (J_P\times J_Q)$ and the $((i_1,i_2),(j_1,j_2))$ coordinate is $0$ if $P_{i_1,j_1}=0$ or $Q_{i_2,j_3}=0$, and otherwise is the entry $(P_{i_1,j_1},Q_{i_2,j_2})$ of the group direct product ${\bf G}\times {\bf H}$.  
This is easily seen, as a tuple in $M^0[{\bf G},P]\times M^0[{\bf H},Q]$ avoiding a $0$-coordinate is of the form $((i_1,g,j_1),(i_2,h,j_2))$ (for $g\in G$, $h\in H$), and the proposed representation $M^0[{\bf G}\times {\bf G},P\times Q]$ is simply considering this tuple rearranged as $((i_1,j_1),(g,h),(j_1,j_2))$.

Let us say that a completely 0-simple semigroup is \emph{block diagonal} if its sandwich matrix may be arranged in block diagonal form (with each rectangular block consisting of strictly nonzero entries).  The following schematic depicts a matrix with $3$ blocks; with large zeroes denoting regions containing only $0$, and grey areas consisting of nonzero entries.
\[
\left[\begin{array}{c|c|c}
\textcolor{black!30}{\rule{0.6cm}{.6cm}}\rule{0cm}{0.7cm}&\rule{0cm}{.6cm}\mbox{\huge{0}}&\rule{0cm}{.6cm}\mbox{\huge{0}}\\\hline
\rule{0cm}{.8cm}\mbox{\huge{0}}&\textcolor{black!30}{\rule{0.8cm}{.8cm}}\rule{0cm}{0.9cm}&\rule{0cm}{.8cm}\mbox{\huge{0}}\\\hline
\rule{0cm}{.5cm}\mbox{\huge{0}}&\rule{0cm}{.5cm}\mbox{\huge{0}}&\textcolor{black!30}{\rule{0.5cm}{.5cm}}\rule{0cm}{0.6cm}
\end{array}\right]
\]
It is not hard to verify that block diagonal completely 0-simple semigroups are precisely those satisfying the law 
\[
xv\neq 0\And uv\neq 0\And uy\neq 0\Rightarrow xy\neq 0.
\]  
{When the number of blocks is 1, the block diagonal property coincides with being a completely simple semigroup.  When there are at least two blocks, the semigroup~${\bf B}_2$ is a subsemigroup: simply generate from any two idempotents selected from separate blocks.}
  
The following result and its proof holds in either the monoid or the semigroup signature.
\begin{cor}\label{cor:orthodoxlike}
Let ${\bf M}=M^0[{\bf G},P]$ be a finite block diagonal completely 0-simple semigroup, where $P$ has at least two blocks \up(equivalently, so that ${\bf M}$ has nontrivial zero divisors\up).  Then the membership problem for finite semigroups or monoids in $\mathsf{V}_{\rm fin}({\bf M}^1)$ is \texttt{NP}-hard.
\end{cor}
\begin{proof}
First observe that ${\bf B}_2^1$ is a submonoid ${\bf M}^1$, so that ${\bf M}^1$ is not in the pseudovariety $\mathsf{L}\mathbf{DS}$.  
Second, we note that ${\bf M}^1$ lies within the join of $\mathsf{V}_{\rm fin}({\bf B}_2^1)$ with the pseudovariety of all finite completely regular semigroups, which is a subclass of $\mathsf{L}\mathbf{DS}$.  
To see this, let $P$ denote the sandwich matrix of ${\bf M}$ in block diagonal form.  
Let $n>1$ denote the number of blocks in $P$ and let $P'$ denote any matrix obtained by replacing $0$ entries in $P$ by group entries (in an arbitrary way).  
The Rees matrix semigroup with $0$ for the sandwich matrix $P'$ has no $0$ divisors, and we denote it by~${\bf R}$.  
If the direct product ${\bf B}_n\times {\bf R}$ is factored by its unique maximal nontrivial ideal, a block diagonal semigroup~${\bf D}$ is obtained whose blocks consist of copies of~$P'$, as this is (isomorphic to) the Kronecker product $P'\times J_n$, where\footnote{We have used $J_n$ for the identity matrix rather than the usual $I_n$, to avoid ambiguity with our instances $I_1,\dots,I_4$.} $J_n$ is the $n\times n$ diagonal sandwich matrix of ${\bf B}_n$. 
So ${\bf M}$ is a subsemigroup of~${\bf D}$: to recover a copy of the $k^{\rm th}$ block of $P$ within $P'\times J_n$, simply restrict to the coordinates of the form $((i,k),(j,k))$, where $(i,j)$ is a coordinate of the $k^{\rm th}$ block in $P$. 
Thus~${\bf M}$ arises as a divisor of ${\bf B}_n\times {\bf R}$, and moreover  ${\bf M}^1$ arises  as a divisor of ${\bf B}_n^1\times {\bf R}^1$.  
As ${\bf R}^1$ is without zero divisors it lies in $\mathsf{L}\mathbf{DS}$ and then as $\mathsf{V}_{\rm fin}({\bf B}_n^1)=\mathsf{V}_{\rm fin}({\bf B}_2^1)$, it follows that ${\bf M}^1$ lies in the join of a subpseudovariety of $\mathsf{L}\mathbf{DS}$ with $\mathbf{V}({\bf B}_2^1)$.  
Now the corollary follows from Theorem~\ref{thm:LDS}.
\end{proof}
All of these results concern \texttt{NP}-hardness, not completeness.  Even the following problem is open.
\begin{problem}\label{prob:B21NP}
Is the membership problem for finite semigroups in the pseudovariety $\mathsf{V}_{\rm fin}({\bf B}_2^1)$ in \texttt{NP}?
\end{problem}

The five element semigroup ${\bf A}_2$ has presentation $\langle a,b\mid aba=a, bab=b, aa=0, bb=b\rangle$, and with an identity element adjoined is denoted ${\bf A}_2^1$.  This is a completely 0-simple monoid lying outside of the class of block diagonal completely 0-simple monoids of Corollary \ref{cor:orthodoxlike}; moreover it lies in the variety of any finite completely 0-simple monoid that is not block diagonal.  In an  article~\cite{jaczha} written subsequently to the present one (but appearing before it), it was shown that $\mathsf{V}_{\rm fin}({\bf B}_2^1)$ is not finitely based within the pseudovariety $\mathsf{V}_{\rm fin}({\bf A}_2^1)$, and that Lemma \ref{lem:is3col} fails to extend to include $\mathsf{V}_{\rm fin}({\bf A}_2^1)$.  So the methods of the present article do not directly shed any light on the complexity of deciding membership in $\mathsf{V}_{\rm fin}({\bf A}_2^1)$.  Thus the following problem is of immediate interest.
\begin{problem}\label{problem:A21}
Can membership in the pseudovariety $\mathsf{V}_{\rm fin}({\bf A}_2^1)$ be decided in polynomial time?
\end{problem}
We wish to make some observations relating to Problem \ref{problem:A21}.  
Recall from McNulty, Szekely and Willard~\cite{MSW} that for a finite algebra ${\bf A}$ of finite signature $\mathscr{S}$, the \emph{equational complexity function} $\beta_{\bf A}:\mathbb{N}\to\mathbb{N}$ is defined by $\beta_{\bf A}(n)=1$ if all $\mathscr{S}$-algebras of size $n$ lie in the variety of ${\bf A}$ and otherwise $\beta_{\bf A}(n)$ is defined to be the maximum, over all $n$-element $\mathscr{S}$-algebras ${\bf B}$ not in the variety of ${\bf A}$, of the smallest equation satisfied by ${\bf A}$ and failing on~${\bf B}$.  
(The precise definition of $\beta_{\bf A}$ depends on the precise definition of ``size'' of an equation, but any sensible will suffice for the following observation.)  
If ${\bf A}$ has a finite basis for its equations, then $\beta_{\bf A}$ is bounded, and the possible converse to this statement remains an important unsolved problem which is known to be equivalent to a well-known open problem of Eilenberg and Sch\"utzenberger \cite{ES}; see \cite{MSW} for a detailed discussion.  
To date there are very few algebras ${\bf A}$ for which $\beta_{\bf A}$ is known not to bounded by a polynomial.  
We wish to observe that ${\bf A}_2^1$ either has non-polynomial bounded equational complexity, or has co-\texttt{NP}-easy membership problem (with respect to many-one reductions), or possibly both.  
The idea is on page 247 of Jackson and McNulty~\cite{jacmcn}: if the equational complexity of ${\bf A}_2^1$ is bounded by a polynomial $p(n)$, then given a finite semigroup ${\bf B}$ not in the variety of ${\bf A}_2^1$, we may guess an equation $u\approx v$ of length at most $p(|B|)$ that is true on ${\bf A}_2^1$, as well as an assignment of its variables into ${\bf B}$ witnessing its failure on ${\bf B}$.  
This certificate for non-membership in the variety of ${\bf A}_2^1$ may be verified in deterministic polynomial time using the algorithm of Seif and Szab\'o~\cite[Theorem~4.7]{seisza} or Klima~\cite[Lemma 8]{kli} to verify that $u\approx v$ holds on ${\bf A}_2^1$ (as well as the easy check that the witnessing assignment into ${\bf B}$ does yield failure of $u\approx v$ on ${\bf B}$).  
We believe this provides reasonable evidence that the membership in the pseudovariety of~${\bf A}_2^1$ is not \texttt{NP}-hard with respect to polynomial time many-one reductions. 

Following Lee and Zhang \cite{leezha} it is known that there are precisely $4$ semigroups of order $6$ that have no finite basis for their equations.  Aside from ${\bf B}_2^1$, which we have shown to generate a pseudovariety with hard membership problem, and ${\bf A}_2^1$ which we have discussed in Problem \ref{problem:A21}, there is the example ${\bf A}_2^g$ observed by Volkov (private communication; see \cite{VGK}) and the semigroup ${\bf L}$ of Zhang and Luo \cite{zhaluo}.  The semigroup ${\bf A}_2^g$ was shown to have polynomial time membership problem for its pseudovariety by Goldberg, Kublanovsky and Volkov \cite{VGK}, but the complexity of membership in $\mathsf{V}_{\rm fin}({\bf L})$ is unknown.
\begin{problem}
Can membership in $\mathsf{V}_{\rm fin}({\bf L})$  be decided in polynomial time?
\end{problem}

At around the same time as the author proved the above results for ${\bf B}_2^1$, Kl\'{\i}ma, Kunc and Pol\'ak distributed a manuscript that gave a family of semigroups for which membership problem in the variety generated by any one is co-\texttt{NP}-complete (with respect to many-one reductions) \cite{KKP}.  Moreover, like ${\bf B}_2^1$, these semigroups are known important examples in the theory of finite semigroups, generating the pseudovariety corresponding to piecewise $k$-testable languages for various $k$.  It is interesting to observe that these semigroups all lie in the pseudovariety $\mathsf{L}{\bf DS}$, thus the direct product of any one of these with ${\bf B}_2^1$ generates a pseudovariety that lies in the interval $[\![\mathsf{V}_{\rm fin}({\bf B}_2^1),\mathsf{L}\mathbf{DS}\vee \mathsf{V}_{\rm fin}({\bf B}_2^1)]\!]$ so has \texttt{NP}-hard membership problem.  It would be very interesting if such a direct product also has co-\texttt{NP}-hard membership problem (with respect to many-one reductions).

\section{Variety membership for other algebraic signatures}\label{sec:small}
To complete the paper, we explore applications of our methods in signatures beyond the semigroup signature.  
We first explore two ad hoc cases producing examples of smallest possible  size, then explore two expansions of the semigroup signature: inverse semigroups and additively idempotent semirings (also known as semilattice-ordered semigroups).  
Some of the results in this section are partially superseded by results in Jackson, Ren and Zhao~\cite{JRZ}, where the instance $I_4$ of the present paper was used in the context of semirings.
This is discussed in more detail in both Subsection~\ref{subsec:small} and Subsection~\ref{subsec:B21inverse}.
\subsection{Variety membership for some small algebraic structures}\label{subsec:small}
An early result due to Lyndon \cite{lyn51} is that  every two element algebra of finite signature has a finite basis for its equations.  Because membership in the variety generated by such an algebraic structure is first order definable, its complexity lies in the complexity class $\texttt{AC}^0$, a proper subclass of~$\texttt{P}$ (the reader should consult a computational complexity text such as \cite{imm} for precise definitions of complexity classes such as $\texttt{AC}^0$).
Thus the smallest possible size for an algebra to generate a variety with intractable membership problem is $3$.  
At the time of submission, the smallest known algebra with intractable membership problem for its variety, assuming $\texttt{P}\neq \texttt{NP}$, had 6 elements: the original example due to Szekely \cite{sze} (subject to minor modification observed in Jackson and McKenzie \cite[p.~123]{jacmck}); the semigroup ${\bf B}_2^1$ of course also has this property by Theorem \ref{thm:LDS}. 
The \texttt{NP}-completeness results above now give an easy example of a $3$-element algebra with \texttt{NP}-hard variety membership problem, as well as a $4$-element groupoid with \texttt{NP}-hard variety membership.  
(Here by \emph{groupoid} we mean an algebraic structure consisting of a single binary operation, also known as a \emph{binar} or a \emph{magma}.)  
Recall that the \emph{graph algebra} $\ga(\mathbb{G})$ of a graph $\mathbb{G}$ is the algebra formed over the vertices $V_\mathbb{G}\cup\{\infty\}$ (where $\infty$ is some new symbol) by letting~$\infty$ act as a multiplicative zero element, and letting $u\cdot v=v$ if $\{u,v\}\in E_\mathbb{G}$ and $u\cdot v=\infty$ otherwise.  
Observe that for vertices $u_1,\dots,u_n\in V_\mathbb{G}$ (not necessarily distinct), we have $(\dots ((u_1\cdot u_2)\cdot u_3\dots )\cdot u_n\neq \infty$ if and only if $\{u_i,u_{i+1}\}\in E_\mathbb{G}$ for each $i\leq n-1$, in which case $(\dots ((u_1\cdot u_2)\cdot u_3\dots )\cdot u_n)=u_n$.  
Thus for a graph~$\mathbb{G}$, if  $u_1,\dots,u_n$ is a sequence of vertices in a path passing each edge at least once, then if we  consider the vertex symbols $u_1,\dots,u_n$ and $x$ as variables, it will follow immediately that $\mathbb{K}_3$ satisfies the law $(\dots ((u_1\cdot u_2)\cdot u_3\dots )\cdot u_n)\approx xx$ if and only if $\mathbb{G}$ is \emph{not} $3$-colourable.
\begin{cor}\label{cor:small}
\begin{enumerate}
\item It is \texttt{NP}-hard to decide membership of finite groupoids in the variety generated by the graph algebra of $\mathbb{K}_3$ \up(with $4$-elements\up).
\item There is a $3$-element algebra ${\bf M}$ with two ternary operations such that it is \texttt{NP}-complete to decide membership of finite algebras in the variety generated by ${\bf M}$.
\end{enumerate}
\end{cor}
\begin{proof}
(1). {We reduce from the \texttt{NP}-hardness result of Corollary \ref{cor:uHK3}, recalling that the property of lying in the universal Horn class of $\mathbb{K}_3$ is identical to being isomorphic to an induced subgraph of a direct power of $\mathbb{K}_3$.   In this situation we will show that $\ga(\mathbb{G}_I)\in\mathsf{V}_{\rm fin}(\ga(\mathbb{K}_3))$.  When $\mathbb{G}_I$ is not $3$-colourable, we will show that $\ga(\mathbb{G}_I)\notin\mathsf{V}_{\rm fin}(\ga(\mathbb{K}_3))$, which will complete the proof.}

First let $\mathbb{G}$ be any graph that is isomorphic to an induced subgraph of a power of~$\mathbb{K}_3$ (such as $\mathbb{G}_I$ when it is a YES instance in Corollary \ref{cor:uHK3}).  We show that $\ga(\mathbb{G})\in \mathsf{V}_{\rm fin}(\ga(\mathbb{K}_3))$.  To see this, assume without loss of generality that $\mathbb{G}$ is an induced subgraph of a power $\mathbb{K}_3^\ell$ of $\mathbb{K}_3$.   Then $\ga(\mathbb{G})$ is a quotient of a subalgebra of $\ga(\mathbb{K}_3)^\ell$: simply take the subalgebra of $\ga(\mathbb{K}_3)^\ell$ generated by the vertices of~$\mathbb{G}$ (which are $\ell$-tuples of vertices in $\mathbb{K}_3$), and factor by the ideal consisting of all elements that have a coordinate equal~$\infty$.

For the converse, assume that $\mathbb{G}$ is a finite connected graph \emph{not} {admitting any homomorphism} into $\mathbb{K}_3$ {(that is, is not $3$-colourable)}.  {This includes, for example, the graph $\mathbb{G}_I$ in Corollary \ref{cor:uHK3} in the case that it is a NO instance.}  Let $u_1,\dots,u_n$ be a sequence of vertices encountered in any path through $\mathbb{G}$ that includes each edge at least once in both directions (such paths obviously exist when $\mathbb{G}$ is connected, and moreover can be constructed in polynomial time, by taking any directed Eulerian circuit in the graph, treating the edge relation as a symmetric binary relation; see \cite{vol}).  Treat these vertex names as variables.  Then $\ga(\mathbb{K}_3)$ satisfies $(\dots ((u_1u_2)u_3)\dots u_n)\approx xx$, while $\ga(\mathbb{G})$ fails this law under the trivial interpretation of variables as vertices.  So $\ga(\mathbb{G})$ does not lie in the variety generated by $\ga(\mathbb{K}_3)$, as required.  

(2).  The idea is very similar, though now we can call directly on the construction of Jackson \cite[Section 7.1]{jac:flat}, which produces a $3$-element algebra $\operatorname{ps}(\mathbbold{2}^\pi)$ from the template $\mathbbold{2}$ for positive 1-in-3SAT. \ 
This algebra  has one ternary operation (corresponding to the ternary relation of $\mathbbold{2}$) and two binary operations $\wedge$, $\triangleright$.  
Proposition~7.2 of~\cite{jac:flat} then shows that the membership problem for finite algebras in the variety of $\operatorname{ps}(\mathbbold{2}^\pi)$ is polynomial time equivalent to the membership problem for the universal Horn class of $\mathbbold{2}$, which is \texttt{NP}-complete by Theorem \ref{thm:SP2}.  
The two binary operations $\wedge$ and $\triangleright$ can be replaced by the single ternary operation $(x\wedge y)\triangleright z$, due to satisfaction of the laws $(x\wedge y)\triangleright y\approx x\wedge y$ and $(x\wedge x)\triangleright y\approx x\triangleright y$.\end{proof}

The aforementioned~\cite{JRZ} used the instance $I_4$ from this paper to give another 3-element algebra
\texttt{NP}-hard membership problem for its pseudovariety, this time an additively idempotent semiring, known as ${\bf S}_7$.

\subsection{${\bf B}_2^1$ in other signatures}\label{subsec:B21inverse}
The semigroup ${\bf B}_2^1$ plays a prominent role in several enriched signatures, most prominently in the inverse semigroup signature $\{\cdot,{}^{-1}\}$ (where $0,1,ab,ba$ are fixed by inverse, and $a^{-1}=b$; see Ka\v{d}ourek \cite{kad} for example) and in the additively idempotent semigroup signature $\{+,\cdot\}$ (see Volkov~\cite{vol21} for example), where addition is idempotent, $1+ab=ab$, $1+ba=ba$ and all other distinct sums equal $0$.  
Both these operations arise naturally from the Wagner-Preston representation of ${\bf B}_2^1$ as a semigroup of injective partial maps, as inverse is just the usual inverse of an injective partial map, while $+$ can be interpreted as the intersection of injective partial maps.
When the signature is clear, we re-use the notation ${\bf B}_2^1$ for  both the inverse semigroup signature variant and the semiring variant, but use  notations such as $\langle B_2^1,\cdot,{}^{-1}\rangle$ (the inverse semigroup case) and so on, for disambiguation when necessary.  

The $3$-element semiring ${\bf S}_7$ from~\cite{JRZ} is the subsemiring of~$\langle B_2^1,+,\cdot\rangle$ on the elements $\{1,a,0\}$, and it was shown that the construction (based on the instance~$I_4$ of the current paper) is strong enough to show \texttt{NP}-hardness of membership in the semiring pseudovariety of ${\bf B}_2^1$ as well, by way of a result~\cite[Theorem~4.6]{JRZ} somewhat reminiscent of Theorem~\ref{thm:LDS} of this paper.
It is interesting to note that there are difficulties encountered in attempting to directly extend the construction of Section~\ref{sec:B21} to the operation~$+$, even if the successful construction built for ${\bf S}_7$ is explicitly modelled on the approach here.

In \cite[Remark 4.8]{JRZ} it is noted that the problem of membership in the pseudovariety of ${\bf S}_7$ is \texttt{NP}-complete.  
The argument for completeness given there does not apply to the semiring ${\bf B}_2^1$,  so  the  semiring version of Problem~\ref{prob:B21NP} is open as well.
\begin{problem}\label{prob:B21NPsemiring}
Is the membership problem for finite semirings in the pseudovariety $\mathsf{V}_{\rm fin}(\langle B_2^1,+,\cdot\rangle)$ in \texttt{NP}?
\end{problem}

In the inverse semigroup signature we arrive at a dual state of knowledge in comparison to the semigroup/monoid signature and the semiring signature: there seems to be significant challenges to applying the methods of this paper to establish \texttt{NP}-hardness, but we are instead able to use the work of Ka\v{d}ourek~\cite{kad} to observe the following upper bound to the computational complexity.
\begin{pro}\label{pro:kadourek}
Membership in the pseudovariety $\mathsf{V}_{\rm fin}(\langle B_2^1,\cdot,{}^{-1}\rangle)$ is in \texttt{NP}.
\end{pro}
Before we prove Proposition~\ref{pro:kadourek}, we need to recall some details of the structural characterisation of the countable members of $\mathsf{V}(\langle B_2^1,\cdot,{}^{-1}\rangle)$ in~\cite{kad}.
For the purposes of computational complexity, we consider only finite inverse semigroups.
It is obvious that every inverse semigroup ${\bf S}$ in the inverse semigroup variety of ${\bf B}_2^1$ must have only trivial subgroups, and it follows from basic semigroup structure theory that $\mathscr{D}=\mathscr{J}$ and every $\mathscr{D}$-class of ${\bf S}$ will be essentially like a Brandt semigroup~${\bf B}_n$: if the inverse subsemigroup generated by the $\mathscr{D}$-class of an element is factored by its unique maximum proper ideal, it will be isomorphic to ${\bf B}_n$ for some $n$.  
The set $\mathscr{D}({\bf S})$ of all $\mathscr{D}$-classes of ${\bf S}$ can be ordered by setting $D\leq D'$ if there is an idempotent $e\in D$ and $f\in D'$ with $ef=e$.  We let the principal filter $\{C\in \mathscr{D}({\bf S})\mid D\leq C\}$ of $D$ in this ordered set be denote $D^\uparrow$ (it is denoted $[D)_S$ in~\cite{kad}).

For any $\mathscr{D}$-classes $C,D$ with $D\leq C$, Ka\v{d}ourek defines two kinds of edge relation, $\alpha_{C,D}$ and $\beta_{C,D}$, using the idempotents of $D$ as vertices (see page~311 of~\cite{kad} for precise definition, noting we are adapting the notation slightly, but in an obvious way).  
For such $C,D$, the two edge relations can be computed in polynomial time, though we omit their precise definitions.
Membership of ${\bf S}$ in the pseudovariety of~${\bf B}_2^1$ is then shown (\cite[Theorem~2.3]{kad}) to be equivalent to a condition of the following form:
for every $C,D\in\mathscr{D}({\bf S})$ with 
$D\leq C$, whenever two  idempotents $e,f$ of $D$ and an idempotent $g$ of $C$ have $eg=e$ and $fg\neq f$, then there is a filter $K$ in $D^\uparrow$  that does \emph{not} include $C$, and is such that 
$f$ is not reachable from~$e$ in the graph whose vertices are the idempotents of~$D$, and whose edge relation is formed by taking 
the $\alpha_{C',D}$-edges from all~$C'$ in~$K$ and the $\beta_{C',D}$-edges from all $C'\in D^\uparrow\backslash K$.
From a computational perspective, there are only polynomially many pairs~$D$ and~$C$, and elements $e,f\in D$ and $g\in C$, with $eg=e$ and $fg\neq f$.  
However there may be exponentially many choices of~$K$, and the defined edge relation changes for each, so a simple brute force check for each such~$K$ will not in general be polynomial time.  
\begin{proof}[Proof of Proposition~\ref{pro:kadourek}]
Given a finite, aperiodic inverse semigroup ${\bf S}$ as input, there are only polynomially many triples $((e,f,g),C,D)$ where $C$ and $D$ are $\mathscr{D}$-classes with $D\leq C$, $e,f$ are idempotents in $D$ and $g$ is an idempotent of $C$ with $eg=e$ and $fg\neq f$.    
A certificate for membership in $\mathsf{V}_{\rm fin}(\langle B_2^1,\cdot,{}^{-1}\rangle)$ is then a list of the pairs $(((e,f,g),C,D),K)$, where $K$ is a filter of $D^\uparrow$ not containing $C$ and such that~$e$ and~$f$ lie in different components of the graph whose edges are formed according to Ka\v{d}ourek's construction.
The certificate check first verifies that the given triples $((e,f,g),C,D)$ are a complete enumeration of all such triples, and then for each of these triples, verifies that $K$ is a filter of $D^\uparrow$ avoiding $C$, and that the  edge relation constructed from $K$ has the required nonreachability condition for $e,f$.  
Each of these polynomially many checks can be performed in polynomial time, so the overall check is polynomial time, and \cite[Theorem~2.3]{kad} ensures that such a certificate exists if and only if the input ${\bf S}$ lies in $\mathsf{V}_{\rm fin}(\langle B_2^1,\cdot,{}^{-1}\rangle)$.
\end{proof}

The methods of the present paper do not easily adapt to show \texttt{NP}-hardness.
The instances ${\bf S}_{I_4}$ are motivated by the situation where $I_4$ is a yes instance: in the proof of Lemma~\ref{lem:isinB21}  we identified ${\bf S}_{I_4}$ in this case as a quotient of a subsemigroup ${\bf T}$ of power of~${\bf B}_2^1$, with the quotient corresponding to factoring a maximal ideal $U$.  
If we consider the inverse subsemigroup generated by the underlying set $T$, applications of inverse create a proliferation of new tuples, and managing a canonical construction that mimics this case becomes difficult. 

\begin{problem}\label{prob:kadourek}
Is membership in $\mathsf{V}_{\rm fin}(\langle B_2^1,\cdot,{}^{-1}\rangle)$ \texttt{NP}-complete?
\end{problem}
This problem is also of interest in the inverse semiring signature $\{+,\cdot,{}^{-1}\}$; see Gusev and Volkov~\cite{gusvol} for explorations of the relationship between these various signatures in the context of ${\bf B}_2^1$.  

The following problem might shed some light on Problem~\ref{prob:B21NP}, in view of the proof of Proposition~\ref{pro:kadourek}.
\begin{problem}\label{prob:dualkadourek}
Is there a criterion similar to Ka\v{d}ourek's for membership of finite semigroups in the semigroup pseudovariety $\mathsf{V}_{\rm fin}(\langle B_2^1,\cdot\rangle)$? 
\end{problem}
A semiring version, and an inverse semiring version of this problem is also of interest.
\medskip

\noindent {\bf Acknowledgement.}
The author wishes to thank the referee, whose meticulous reading of the original submission identified many deficiencies; the article has benefitted significantly from the resulting improvements.
The author wishes to thank Prof.~Mikhail Volkov for numerous suggestions including the exploration of the work of 
Ka\v{d}ourek's~\cite{kad} in Subsecton~\ref{subsec:B21inverse}.

\bibliographystyle{amsplain}

\end{document}